\documentclass[12pt,a4]{article}
\usepackage{amsmath,amssymb,amsthm}
\usepackage[
colorlinks=true,
linkcolor=blue,
filecolor=blue,
urlcolor=blue
]{hyperref}
\AtBeginDvi{}
\usepackage{color}
\theoremstyle{plain}
\newtheorem{thm}{Theorem}[section]
\newtheorem{lem}[thm]{Lemma}

\newtheorem{prop}[thm]{Proposition}
\newtheorem{defn}[thm]{Definition}
\newtheorem{fact}[thm]{Fact}
\newtheorem{example}[thm]{Example}

\newtheorem{rem}[thm]{Remark}
\newtheorem{obs}[thm]{Observation}
\newcommand{\End}{\mathop{\mathrm{End}}\nolimits}
\newcommand{\diag}{\mathop{\mathrm{diag}}\nolimits}
\newcommand{\len}{\mathop{\mathrm{len}}\nolimits}
\newcommand{\qbinom}[2]{\genfrac{[}{]}{0pt}{}{#1}{#2}} 
\def\set#1#2{\{#1\,|\,#2\}}

\newcommand{\C}{\mathbb{C}}
\newcommand{\R}{\mathbb{R}}
\newcommand{\Z}{\mathbb{Z}}

\newcommand{\Partition}[1]{\mathcal{P}_{#1}}
\newcommand{\G}[2]{\mathcal{G}^\mathbb{C}_{#1,#2}}
\newcommand{\Range}[2]{\mathrm{Range}(\G{m}{n})}
\newcommand{\Ades}{\mathcal{E}}
\newcommand{\Adesh}{\mathcal{F}}
\newcommand{\inv}{{}^{-1}}
\newcommand{\fre}{\mathfrak{e}}
\newcommand{\frh}{\mathfrak{h}}
\newcommand{\vspan}{\text{-}\mathrm{span}}
\begin{document}

\title{Great antipodal sets of complex Grassmannian manifolds as designs with the smallest cardinalities}
\author{Hirotake Kurihara, Takayuki Okuda}
\date{}

\maketitle

\renewcommand{\thefootnote}{\fnsymbol{footnote}}

\begin{abstract}
The aim of this paper is a characterization of great antipodal sets of complex Grassmannian manifolds as certain designs with the smallest cardinalities.
\end{abstract}


\section{Introduction}
\label{sec:intro}

The aim of this paper is a characterization of great antipodal sets of 
complex Grassmannian manifolds as certain designs with the smallest cardinalities in Theorem~\ref{Mainthm:great_anti_to_EF_design}.
Note that great antipodal sets are researched in the area of differential geometry.
On the other hand, the theory of designs is related to algebraic combinatorics or representation theory.

In 1973, Delsarte~\cite{Delsarte1973aat}
unified the theories of codes and designs on association schemes,
and
gave the upper bounds for codes and 
the lower bound for designs by applying linear programming
for polynomials associated with metric or cometric association schemes.
After his work, 
the theory of spherical designs was introduced by Delsarte--Goethals--Seidel~\cite{Delsarte1977sca}
as an analogy of Delsarte technique.
The essential tool in their works is the addition formula
for polynomials;
polynomials associated with metric or cometric association schemes,
or the Gegenbauer polynomials with spheres.

In general, the theory of designs can be given on the Delsarte
spaces (See Neumaier~\cite{Neumaier1980dgd}, Godsil\cite{Godsil1993ac})
or the polynomial spaces (See Levenshtein~\cite{Levenshtein1992dam,
 Levenshtein1998oti, Levenshtein1998ubf}),
which are metric spaces with ``good'' polynomials, such as the polynomials associated with metric or cometric association schemes,
or the Gegenbauer polynomials. 
For instance, compact symmetric spaces of rank one are natural and significant examples of the Delsarte spaces or the polynomial spaces for continuous metric spaces.
Note that spheres and projective spaces are compact symmetric spaces of rank one.
The theory of designs on compact symmetric spaces of rank one
was studied in details by Hoggar~\cite{Hoggar1982tip}.
Moreover, the classification problems of ``tight designs'',
which is a design whose cardinality is equal to the known natural lower bound, were developed by Bannai--Hoggar~\cite{Bannai1985ott,Bannai1989tta},
Hoggar~\cite{Hoggar1984pot, Hoggar1989tad}
and Lyubich~\cite{Lyubich2009otp}, and others.

Unfortunately, it is not easy to consider 
similar theories in the framework of compact symmetric spaces of higher rank.
Nevertheless, there are some developments in the theory of designs on Grassmannian manifolds, which are natural examples of compact symmetric spaces of higher rank
(See Bachoc--Coulangeon--Nebe~\cite{Bachoc2002dig},
Bachoc--Bannai--Coulangeon~\cite{Bachoc2004cad}
and Roy~\cite{Roy2009bca}).
We refer the readers to the survey by Bannai-Bannai~\cite{Bannai2009aso}
for more information of the history of generalizations of the theory of designs.

In this paper,
our interested is to give a generalization of the following well known fact:
a pair of antipodal points of a sphere,
which is purely differential geometric notion,
can be characterized by a tight $1$-design,
which is purely algebraic combinatorial notion.

The concept of the antipodal points in spheres is generalized to antipodal sets in symmetric spaces by Chen--Nagano~\cite{Chen1988arg}.
An antipodal set whose cardinality is maximal in the set of antipodal sets is called a great antipodal set. 
Since it is known that any antipodal set is finite, 
a great antipodal set is a finite set.
If our space is a symmetric $R$-space,
then any two great antipodal sets are congruent.
Thus, a great antipodal set is unique in the sense of congruent
(See Tanaka--Tasaki~\cite{Tanaka2013aos}).

Complex Grassmannian manifolds are important examples of symmetric $R$-spaces.
Here the complex Grassmannian manifold $\G{m}{n}$ is defined to be
the set of $m$-dimensional subspaces in the $n$-dimensional complex vector space $\C^n$.
A great antipodal set $S$ on $\G{m}{n}$ consists of $\binom{n}{m}$ points.
$S$ is unique up to the action of the unitary group $U(n)$.
In this paper,
we are interested in the following question:
\begin{description}
\item[(Q)]
Great antipodal sets of the complex Grassmannian manifolds
can be considered as ``good'' designs?
\end{description}

The concept of designs on $\G{m}{n}$ was introduced by Roy in 2009.
In this paper, to capture the feature of great antipodal sets
in terms of designs,
we modify Roy's definition of designs on $\G{m}{n}$
to more suitable one by using indexes of irreducible
representations of unitary groups in Section~\ref{sec:design}.
To be more detail, we define a $\mathcal{T}$-design
on $\G{m}{n}$ for a subset $\mathcal{T}$ of
\[
\Partition{m} := \{\, (\mu_1,\dots,\mu_m) \in \Z^m \mid \mu_1 \geq \dots \geq \mu_m \geq 0 \,\}.
\]

As a main result of this paper, 
we show that a great antipodal set is a $\Ades \cup \Adesh$-design, where
\[
\Ades:=
\set{(\underbrace{1,1,\ldots ,1}_{i},\underbrace{0,0,\ldots ,0}_{m-i})}
{i=0,1,\ldots ,m},
\]
\[
\Adesh:=
\set{(2,\underbrace{1,1,\ldots ,1}_{i-1},\underbrace{0,0,\ldots ,0}_{m-i})}
{i=2,\ldots ,m}\subset \Partition{m}.
\]
We also give the lower bound for the cardinality of a $\Ades$-design 
$X$ as follows:
\[
|X| \ge \binom{n}{m}
\]
(See Theorem~\ref{Mainthm:great_anti_to_E_design}).
In particular, any $\Ades \cup \Adesh$-design also satisfies the inequality above.
Remark that this lower bounds is attained by a great antipodal set.
In addition, we show that a $\Ades \cup \Adesh$-design $X$
satisfying $|X|=\binom{n}{m}$
must be a great antipodal set.
In other words, the property of $\Ades \cup \Adesh$-designs
with the smallest cardinalities
gives a characterization of great antipodal sets (See Theorem~\ref{Mainthm:great_anti_to_EF_design}).

The outline of this paper is as follows. In Section \ref{sec:Preliminary}, we recall the definition and some
properties of great antipodal sets of the complex Grassmannian manifolds.
In Section \ref{sec:design}, we give a definition of designs on
the complex Grassmannian manifolds and describe our main results.
In Section \ref{sec:zonal_poly},
in order to prove Theorem~\ref{Mainthm:great_anti_to_E_design}
and \ref{Mainthm:great_anti_to_EF_design},
we recall zonal orthogonal polynomials and show some properties
of them.
In Section \ref{sec:proof}, we give proofs of our results.
As an appendix, we give a lower bound for $1$-designs on $\G{m}{n}$
in Appendix~\ref{sec:appendix},
which the concept of $t$-designs on $\G{m}{n}$ was introduced by
Roy~\cite{Roy2009bca}.
In particular, we determine $1$-designs on $\G{m}{n}$ with the smallest cardinalities when $n$ is divided by $m$.
Finally, in Appendix \ref{sec:appendix2}, 
we show an example of $\Ades$-design 
with the smallest cardinality
which is not a great antipodal set of $\G{m}{n}$.

\section{Preliminary for antipodal sets}
\label{sec:Preliminary}

\subsection{Notation for complex Grassmannian manifolds}
\label{subsection:notations_CG}

We set up notation for complex Grassmannian manifolds in this subsection.

We fix positive integers $n$ and $m$ with $m < n$.
Let us consider the $n$-dimensional complex vector space $\mathbb{C}^n$ 
equipped with the standard Hermitian inner product.
We denote the \emph{complex Grassmannian manifold of rank $m$} by 
\[
\G{m}{n} := \{\, \text{$m$-dimensional complex subspaces of $\mathbb{C}^n$} \,\}.
\]
For $a\in \G{m}{n}$, let us denote $a^\perp \subset \C^n$
the orthogonal complement ($(n-m)$-dimensional) subspace of
$\C^n$ for $a$.
Since the map 
$\G{m}{n}\rightarrow \G{n-m}{n}$, $a \mapsto a^\perp$
is a diffeomorphism, we always assume $n \ge 2m$.

The complex Grassmaniann manifold $\G{m}{n}$ is a compact Riemannian symmetric space.
In fact, $\G{m}{n}$ can be represented as 
a homogeneous space $G/K$ 
by a compact symmetric pair $(G,K)$ 
defined as follows.
Let us put \[
G := U(n) = \{\, g \in M(n,\mathbb{C}) \mid g^* = g\inv \,\},
\]
where $g^*$ denotes the conjugate transpose of the complex matrix $g$.
Note that $G = U(n)$ is a compact Lie group.
The natural representation of $G = U(n)$ on $\mathbb{C}^n$
induces a transitive $G$-action on $\G{m}{n}$.
We write $\{ e_1,\dots,e_n \}$ for the standard orthonormal basis of $\mathbb{C}^n$,
and take $a_0 := \mathbb{C}\vspan \{ e_1,\dots,e_m \} \in \G{m}{n}$.
Then the isotropy subgroup $K$ at $a_0 \in \G{m}{n}$ can be written by 
\begin{align*}
K 	&:= \left\{ \begin{pmatrix} g_m & \\ & g_{n-m} \end{pmatrix} \mid g_m \in U(m),\ g_{n-m} \in U(n-m) \right\} \\
	&\simeq U(m) \times U(n-m).	
\end{align*}
Therefore, $\G{m}{n}$ can be represented as $G/K \simeq U(n)/U(m) \times U(n-m)$.
Here, we define an involutive automorphism $\tau$ on $U(n)$ by 
\[
\tau : U(n) \rightarrow U(n),\quad g \mapsto I_{m,n-m} \cdot g \cdot I_{m,n-m},
\]
where 
\[
I_{m,n-m} := \begin{pmatrix} I_m & \\ &  -I_{n-m} \end{pmatrix} \in M(n,\mathbb{R}).
\]
Then we can observe that the subgroup $K$ of $G$ 
consisted of all fixed points of $\tau$,
i.e.~
\[
K = \{\, g \in U(n) \mid \tau(g) = g \,\}.
\]
Therefore, $(G,K)$ is a compact symmetric pair,
and hence,
$\G{m}{n} \simeq G/K$ 
is a compact Riemannian symmetric space.

\subsection{Antipodal sets on a symmetric $R$-space}\label{subsection:pre:antipodal_on_symmR}

Let $\Omega$ be a connected compact Riemannian symmetric space.
For each point $x$ of $\Omega$, 
we denote $s_x : \Omega \rightarrow \Omega$ the geodesic symmetry at $x$. 
A subset $S$ of $\Omega$ is said to be \textit{antipodal} if $s_x(y) = y$ for any $x,y \in S$. 
Since $\Omega$ is compact and 
each $x \in \Omega$ is an isolated fixed point of the symmetry $s_x$,
any antipodal set $S$ must be finite.
Chen--Nagano \cite{Chen1988arg} defined the $2$-number of $\Omega$ by 
\[
\sharp_2 \Omega :=  \max \{\, |S| \mid S \text{ is an antipodal set in $\Omega$} \,\},
\]
and an antipodal set $S$ is said to be \textit{great} if $|S| = \sharp_2 \Omega$.

Our interesting in this paper is in the case where 
$\Omega$ is the complex Grassmannian manifold 
$\G{m}{n} \simeq U(n)/{U(m) \times U(m-n)}$.
It is well known that $\G{m}{n}$
is 
a Hermitian symmetric space,
and hence, 
a symmetric $R$-space (See \cite[\S II.1 in Part II]{Faraut-Kaneyuki-Koranyi-Ku-Roos} for the definition of symmetric $R$-spaces).

We recall fundamental results for antipodal sets on a symmetric $R$-space $\Omega = G/K$ as follows:

\begin{fact}[Takeuchi \cite{Takeuchi1989tn}]\label{fact:Takeuchi_2-number}
Let $\Omega$ be a symmetric $R$-space.
Then 
\[
\sharp_2 \Omega = \dim H_*(\Omega,\mathbb{Z}_2),
\]
where $H_*(\Omega,\mathbb{Z}_2)$ is the homology group of $\Omega$ with coefficient $\mathbb{Z}_2$.
\end{fact}

\begin{fact}[S\'{a}nchez \cite{Sanchez1997ind}, Tanaka--Tasaki \cite{Tanaka2013aos}]\label{fact:uniqueness_great}
Let $\Omega \simeq G/K$ be a symmetric $R$-space.
Then the following holds$:$
\begin{enumerate}
\item Any maximal antipodal set on $\Omega$ is great.
That is,
for any antipodal set $S_0$ of $\Omega$, 
there exists a great antipodal set $S$ of $\Omega$ such that $S_0 \subset S$. 
\item A great antipodal set of $\Omega$ is unique up to the conjugation of $G$-action. 
That is, for any two great antipodal sets $S$ and $S'$, there exists $g \in G$ such that $S' = g S$.
\end{enumerate}
\end{fact}

\subsection{Antipodal sets on a complex Grassmannian manifold}\label{subsection:antipodal_on_CG}

We go back to the case where $\Omega$ is a complex Grassmannian manifold $\G{m}{n}$.

To notice antipodal sets on $\G{m}{n}$,
we recall the geodesic symmetry $s_a$ on $\G{m}{n}$ for each $a \in \G{m}{n}$ 
as follows.
Let us fix $a \in \G{m}{n}$.
Since $\C^n$ is decomposed as the orthogonal direct sum
$\C^n = a \oplus a^\perp$,
for each vector $v \in \mathbb{C}^n$,
there uniquely exist $v_a \in a$ and $v_{a^\perp} \in a^\perp$ with $v = v_a + v_{a^\perp}$.
Let us consider the involutive operator $\tilde{s}_a$ on $\mathbb{C}^n$
acts on $a$ as 
the identity 
and 
acts on $a^\perp$ as $-1$.
That is, 
\begin{align}
\tilde{s}_a(v) = v_a - v_{a^\perp} \quad \text{for } v \in \mathbb{C}^n. \label{def:tilde-s}
\end{align}
The involutive operator $\tilde{s}_a$ on $\mathbb{C}^n$ induces an involutive isometry $s_a$ on $\G{m}{n}$.
It is known that the map $s_a : \G{m}{n} \rightarrow \G{m}{n}$ is the geodesic symmetry at $a \in \G{m}{n}$.
In particular, for fixed $a,b \in \G{m}{n}$,
the equality $s_a(b) = b$ holds if and only if 
$b$ is decomposed as the orthogonal direct sum
\[
b = (a \cap b) \oplus (a^\perp \cap b).
\]

We give an example of antipodal set as follows.
Recall that $\{ e_1,\dots,e_n \}$ denotes the standard orthonormal basis of $\mathbb{C}^n$.
For each $m$-subset $I$ of $\{\, 1,\dots, n \,\}$,
we put $a_{I} := \mathbb{C}\vspan \{\, e_{i} \mid i \in I \,\} \in \G{m}{n}$.
Then, for any $m$-subsets $I$ and $I'$ of $\{\,1,\dots,n\,\}$,
one can easily observe that 
$a_{I'}$ is decomposed as the orthogonal direct sum
$a_{I'} = (a_{I} \cap a_{I'}) \oplus (a_{I}^\perp \cap a_{I'})$. 
Hence, 
we have $s_{a_{I}}(a_{I'}) = a_{I'}$.
Therefore, the finite subset 
\begin{align}
S := \{\, a_{I} \mid \text{$I$ is an $m$-subset of $\{\, 1,\dots,n \,\}$}\,\} \subset \G{m}{n} \label{def:standard_great_antipodal}
\end{align}
is an antipodal set of $\G{m}{n}$ with $|S| = \binom{n}{m}$.
One can observe that $S$ is a maximal antipodal set on $\G{m}{n}$.
Thus, by Fact \ref{fact:uniqueness_great}, 
the antipodal set $S$ on $\G{m}{n}$ is great 
and any great antipodal set $S'$ on $\G{m}{n}$ 
is congruent to $S$ by the $U(n)$-action.
In particular, the following holds:
\begin{fact}[Chen--Nagano \cite{Chen1988arg}]\label{fact:2-number_of_Gmn}
$\sharp_2 \G{m}{n} = \binom{n}{m}$.
\end{fact}

Remark that a great antipodal set $S$ carries a metric and cometric association scheme.
For detail of association schemes, see Bannai--Ito~\cite{Bannai1984aci}
and Brouwer--Cohen--Neumaier~\cite{Brouwer1989dg}.

\section{Main results}
\label{sec:design}

It is known that 
for any $1$-design $X$ on the $(n-1)$-dimensional complex projective space, 
the inequality $|X| \geq n$ holds, and $X$ is said to be tight if $|X| = n$
(See Hogger \cite{Hoggar1982tip} for more details).
It is also known, but implicitly, that 
for a complex projective space,
great antipodal sets 
can be characterized by tight $1$-designs. 
That is, the following fact holds:

\begin{fact}\label{fact:projective}
For a finite subset $S$ of the $(n-1)$-dimensional complex projective space,
the following are equivalent:
\begin{enumerate}
\item $S$ is a great antipodal set, 
i.e.,  $a \perp b$ for any $a,b \in S$ with $a \neq b$
and $|S| = n$.
\item $S$ is a tight $1$-design, 
i.e., $S$ is an $1$-design with $|S| = n$.
\end{enumerate}
\end{fact}

In this section,
we define designs on $\G{m}{n}$
by using some basic facts for harmonic analysis,
and give a kind of generalization of Fact \ref{fact:projective} to 
complex Grassmannian manifolds $\G{m}{n}$.

\subsection{Harmonic analysis on $\G{m}{n}$}
\label{subsection:Harmonic}

In this subsection,
we briefly recall some basic facts for harmonic analysis on $\G{m}{n}$
and set up our notation. 
A detailed summary of harmonic analysis on $\G{m}{n}$ can be found in
Roy~\cite{Roy2009bca}.

Let us put $C^0(\G{m}{n})$ to the functional space on $\G{m}{n}$
consisted of all $\mathbb{C}$-valued continuous functions on $\G{m}{n}$.
Recall that $U(n)$ acts on $\G{m}{n}$ transitively
with the isotropy subgroup $K \simeq U(m) \times U(n-m)$ at 
\[
a_0 := \mathbb{C}\vspan \{e_1,\dots,e_m \} \in \G{m}{n},
\]
where 
$e_1,\dots,e_{n}$ 
is the standard basis of $\mathbb{C}^n$,
see Section \ref{subsection:notations_CG}. 
Then $U(n)$ acts also on $C^0(\G{m}{n})$ naturally.
Here, we denote by $\nu$ 
the normalized $U(n)$-invariant Haar measure on $\G{m}{n}$,
where ``normalized'' means that $\nu(\G{m}{n}) = 1$.
Then $C^0(\G{m}{n})$ has the $L^2$ inner product 
with respect to the measure $\nu$:
\begin{align}
\langle f,g \rangle:=\int_{\G{m}{n}} (f \cdot \overline{g}) d\nu \label{def:innerproduct}
\end{align}
for $f,g\in C^0(\G{m}{n})$,
where $\overline{g}$ and $f \cdot \overline{g}$ denotes the complex conjugation of the function $g$ and the product of $f$ and $\overline{g}$, respectively. 
Since $\nu$ is $U(n)$-invariant,
the infinite-dimensional representation of $U(n)$ on $C^0(\G{m}{n})$ 
preserves the inner product on $C^0(\G{m}{n})$.

By the highest weight theory,
a complex irreducible unitary representation of a connected compact Lie group
is determined by 
its highest weight, up to isomorphisms.
For the unitary group $U(n)$,
the set $\widehat{U(n)}$ of 
the complex irreducible unitary representations of $U(n)$, 
up to isomorphisms,
can be regarded as
\[
\widehat{U(n)} \simeq \{\, (\lambda_1,\dots,\lambda_n) \in \mathbb{Z}^n \mid \lambda_1 \geq \cdots \geq \lambda_n \,\}
\]
as in Bump~\cite[Theorem 38.3]{Bump2004Lie}.
Throughout this paper, for each $\lambda \in \widehat{U(n)}$, 
we fix an irreducible unitary representation $(\rho_\lambda,V_{\lambda})$ of $U(n)$ 
corresponding to $\lambda$.

We also put 
\begin{align}
\Partition{m} := \{\, (\mu_1,\dots,\mu_m) \in \Z^m \mid \mu_1 \geq \dots \geq \mu_m \geq 0 \,\} \label{def:Pm}
\end{align}
and define the map $\phi: \Partition{m} \rightarrow \widehat{U(n)}$ by 
\[
\phi : \Partition{m} \rightarrow \widehat{U(n)},\ \mu \mapsto (\mu_1,\dots,\mu_m, 0, \dots,0,-\mu_m,\dots,-\mu_1).
\]
Then it is known that
$
\dim V_{\phi(\mu)}^{K} = 1
$
for any $\mu \in \Partition{m}$,
where 
\[
V_{\lambda}^{K} := \{\, v \in V \mid \rho_\lambda(k)v=v \ \text{for any } k \in K \,\}
\]
(See Goodman--Wallach~\cite[\S 12.3.2]{Goodman-Wallach2000rep}
for the details).

Throughout this paper, 
we use the following notation for elements of $\Partition{m}$:
\begin{align}
(1^i) &:= (\underbrace{1,1,\ldots ,1}_{i},\underbrace{0,0,\ldots ,0}_{m-i}), \label{def:1^i} \\
(i) &:= (i,\underbrace{0,0,\ldots ,0}_{m-1}), \label{def:i} \\
(2,1^{i-1}) &:= (2,\underbrace{1,1,\ldots ,1}_{i-1},\underbrace{0,0,\ldots ,0}_{m-i}). \label{def:2_1^i-1}
\end{align}

\begin{example}
By using Weyl's character formula
(cf.~Sepanski~\cite[Theorem 7.28]{Sepanski2007clg}),
one can prove that 
\[
\bigoplus_{i=0}^m V_{\phi((1^i))} \simeq \End_\mathbb{C} (\wedge^m \mathbb{C}^n)
\]
as representations of $U(n)$.
\end{example}

Let us fix $\mu \in \Partition{m}$.
For each $v \in V_{\phi(\mu)}$, $w \in V_{\phi(\mu)}^K$ 
and $a \in \G{m}{n}$, we put 
\[
\Phi_{\mu}(v \otimes w) (a) := \langle v,g_a w \rangle_{\phi(\mu)}, 
\]
where $g_a$ is an element $g_a \in U(n)$ with $g_a \cdot a_0 = a$
and 
$\langle \cdot,\cdot \rangle_{\phi(\mu)}$ denotes 
the Hermitian inner product of $V_{\phi(\mu)}$.
Then $\Phi_{\mu}(v \otimes w) (a)$ is well-defined
and $\Phi_{\mu}$ can be extended to a $\mathbb{C}$-linear map 
\[
\Phi_{\mu} : V_{\phi(\mu)} \otimes V_{\phi(\mu)}^{K} \rightarrow C^0(\G{m}{n}).
\]
We put 
\begin{align}
H_\mu := \Phi(V_{\phi(\mu)} \otimes V_{\phi(\mu)}^K) \subset C^0(\G{m}{n}).
\end{align}
Then $H_\mu$ is a finite dimensional subspace of $C^0(\G{m}{n})$.

\begin{example}\label{ex:const}
For the case where $\mu = (0) \in \Partition{m}$, 
the functional space $H_{(0)}$ is consited of all constant functions on $\G{m}{n}$,
see \eqref{def:i} for the notation of $(0) \in \Partition{m}$.
In particular, $\dim H_{(0)} = 1$.
\end{example}

By the Peter--Weyl theorem for compact symmetric spaces, 
which can be found in Takeuchi~\cite[Theorem 1.3]{Takeuchi1975spherical}, 
we obtain the following facts:

\begin{fact}\label{fact:func_realization_of_reps}
The following holds:
\begin{enumerate}
\item For any $\mu \in \Partition{m}$,
the space $H_\mu$ is the unique $U(n)$-subrepresentation of $C^0(\G{m}{n})$ with 
$H_\mu \simeq V_{\phi(\mu)}$ as unitary representations of $U(n)$.
In particular, $\dim H_\mu = \dim V_{\phi(\mu)}$.
\item 
If $\mu \neq \mu'$ in $\Partition{m}$, then $H_\mu \perp H_{\mu'}$ in $C^0(\G{m}{n})$.
\item \label{item:PW} The subspace $\bigoplus_{\mu \in \Partition{m}} H_\mu$ is dense in $C^0(\G{m}{n})$.
\end{enumerate}
\end{fact}

Note that 
Fact \ref{fact:func_realization_of_reps} \eqref{item:PW} gives 
the irreducible decomposition of 
the $U(n)$-representation $C^0(\G{m}{n})$.

It is known that the dimension of the irreducible representation of a compact Lie group can be computed by Weyl's dimension formula
(cf.~Sepanski~\cite[Theorem 7.32]{Sepanski2007clg}).
In the case of $U(n)$,
the dimension of an irreducible representation $V_{\lambda}$
is
\[
\dim V_{\lambda} =
\prod_{1\le i<j\le n}\frac{\lambda_i-\lambda_j+j-i}{j-i}.
\]
This formula yields the following:
\begin{eqnarray}
\label{eq:dimH_(1^i)}
\dim H_{(1^i)}&=&\frac{n-2i+1}{n+1}\binom{n+1}{i}^2,\\
\label{eq:dimH_(i)}
\dim H_{(i)}&=&\frac{n+2i-1}{n-1}\binom{n+i-2}{i}^2,\\
\label{eq:dimH_(2,1^i-1)}
\dim H_{(2,1^{i-1})}&=&\frac{i^2 (n+3) (n-2i+1)}{(n-i+2)^2}\binom{n+1}{i+1}^2,
\end{eqnarray}
see \eqref{def:1^i}, \eqref{def:i} and \eqref{def:2_1^i-1} for the notation of $(1^i)$, $(i)$ and $(2,1^{i-1})$.

\subsection{Designs on $\G{m}{n}$}\label{subsection:definition_of_designs}

For a finite subset $\mathcal{T}$ of $\Partition{m}$, 
we put 
\begin{align}
H_{\mathcal{T}} :=
\bigoplus_{\mu \in \mathcal{T}} H_\mu. \label{def:H_T}
\end{align}
For the case where $\mathcal{T} = \emptyset$, then we put 
$
H_{\emptyset} := \{ 0 \}.
$

We give a definition of $\mathcal{T}$-designs on $\G{m}{n}$ as follows:

\begin{defn}\label{defn:T-design}
Let $X$ be a non-empty finite subset of $\G{m}{n}$,
and $\mathcal{T}$ be a finite subset of $\Partition{m}$.
We say that $X$ is a $\mathcal{T}$-design if 
\[
\frac{1}{|X|}\sum_{a \in X} f(a) = \int_{\G{m}{n}} f d\nu \quad \text{for any } f \in H_\mathcal{T}.
\]
\end{defn}

Here are three easy observations for designs on $\G{m}{n}$ as follows:

\begin{obs}\label{obs:designs}
Let $\mathcal{T}$, $\mathcal{T}'$ be finite subsets in $\Partition{m}$.
Then the following holds:
\begin{itemize}
\item If $X \subset \G{m}{n}$ is both a $\mathcal{T}$-design and a $\mathcal{T}'$-design.
Then $X$ is also a $\mathcal{T} \cup \mathcal{T}'$-design.
\item If $\mathcal{T}' \subset \mathcal{T}$,
then any $\mathcal{T}$-design is also a $\mathcal{T}'$-design.
\item If $\mathcal{T} = \{(0)\}$, then 
any non-empty finite subset of $\G{m}{n}$ is an $\mathcal{T}$-design on $\G{m}{n}$.
\end{itemize}
\end{obs}

\begin{rem}
Let $t$ be a non-negative integer.
As an analogy of the concept of
$t$-designs on rank one symmetric spaces,
the concept of $t$-designs on $\G{m}{n}$ was introduced by
Roy~\cite{Roy2009bca}.
Remark that a $t$-design on $\G{m}{n}$ in terms of Roy's definition
is translated as a $\mathcal{T}_t$-design on $\G{m}{n}$,
where $\mathcal{T}_t:=\set{\mu\in \Partition{m}}
{\sum^m_{i=1}\mu_i\le t}$.
\end{rem}

\subsection{A characterization of great antipodal sets of $\G{m}{n}$}

In this subsection,
we will give a characterization of 
great antipodal sets on $\G{m}{n}$
as $\mathcal{T}$-designs 
with the smallest cardinality
for a suitable $\mathcal{T}$
in Theorem \ref{Mainthm:great_anti_to_EF_design}.
Recall that 
a great antipodal set $S$ of $\G{m}{n}$
is congruent to that given as \eqref{def:standard_great_antipodal} in Section \ref{subsection:antipodal_on_CG}.
In particular, $|S| = \binom{n}{m}$.

To describe our main results,
we put 
\begin{align}
\Ades := 
\set{(1^i)}{i=0,1,\ldots ,m} \subset \Partition{m},
\end{align}
see \eqref{def:1^i} for notation of $(1^i)$.

The following is our first main result:

\begin{thm}
\label{Mainthm:great_anti_to_E_design}
A great antipodal set $S$ of $\G{m}{n}$ is an $\Ades$-design with the smallest cardinality.
\end{thm}

The proof of Theorem \ref{Mainthm:great_anti_to_E_design}
will be given in Section \ref{sec:proof}.

Do there exist $\Ades$-designs with $\binom{n}{m}$ nodes
without great antipodal sets?
We give an example of such an $\Ades$-design in Appendix \ref{sec:appendix2}.

To give a characterization of great antipodal sets of $\G{m}{n}$,
we also put 
\begin{align}
\Adesh :=
\set{(2,1^{i-1})}
{i=2,\ldots ,m}\subset \Partition{m},
\end{align}
see \eqref{def:2_1^i-1} for notation of $(2,1^{i-1})$.

In Section \ref{sec:proof}, 
we will prove the following theorem, 
which gives a characterization of great antipodal sets
as certain designs with the smallest cardinality:

\begin{thm}\label{Mainthm:great_anti_to_EF_design}
Let $S$ be a finite subset of $\G{m}{n}$.
Then the following two conditions on $S$ are equivalent:
\begin{enumerate}
\item $S$ is an antipodal set of $\G{m}{n}$ with the largest cardinality,
i.e., a great antipodal set of $\G{m}{n}$.
\item $S$ is a $\Ades \cup \Adesh$-design on $\G{m}{n}$ with the smallest cardinality.
\end{enumerate}
\end{thm}

\begin{rem}
When $m=1$, 
the symbol $\G{1}{n}$ denotes the $(n-1)$-dimensional complex projective space.
In this case, 
the concept of $\Ades \cup \Adesh$-designs on $\G{1}{n}$ 
is the same to 
the concept of $1$-designs on the $(n-1)$-dimensional complex projective space.
Therefore, Theorem \ref{Mainthm:great_anti_to_EF_design} is a generalization of Fact \ref{fact:projective}.
In Appendix \ref{sec:appendix},
when $n$ is divided by $m$,
we also determine the properties of
$1$-designs on $\G{m}{n}$,
in the sense of Roy \cite{Roy2009bca},
with the smallest cardinality.
\end{rem}

\section{Zonal orthogonal polynomials}
\label{sec:zonal_poly}

\subsection{Principal angles}\label{subsection:main-angle}

For an element $a \in \G{m}{n}$,
we denote by $P_a \in \End(\C^n)$ the orthogonal projection on $\C^n$ to $a$.
It is known that 
for two elements $a$, $b$ in $\G{m}{n}$, 
any eigenvalue of $P_a \circ P_b \in \End(\C^n)$ is in the real interval $[0,1]$ and the number of non-zero eigenvalues is at most $m$.
That is,
there uniquely exists real numbers $y_1(a,b),\dots, y_m(a,b)$ 
with $1 \geq y_1(a,b) \geq \dots \geq y_m(a,b) \geq 0$ such that 
the linear operator $P_a \circ P_b \in \End(\C^n)$ 
can be considered as 
$\diag(y_1(a,b),\dots,y_m(a,b),0,\dots,0) \in M(n,\mathbb{C})$
with respect to a suitable basis of $\mathbb{C}^n$.
Hereafter, we put $y(a,b) := (y_1(a,b),\dots,y_m(a,b))$, 
which is called the \emph{principal angles between $a$ and $b$}.
One can observe that the principal angles are symmetric, i.e.,
$y(a,b)=y(b,a)$ for $a,b\in \G{m}{n}$.
The details can be referred to Roy~\cite{Roy2009bca}.

Let us write 
\begin{align*}
\Range{m}{n} &:= \{\, (y_1,\dots,y_m) \in \mathbb{R}^m \mid 1 \geq y_1 \geq \dots \geq y_m \geq 0 \,\}.
\end{align*}
Then the map 
\[
y: \G{m}{n} \times \G{m}{n} \rightarrow \Range{m}{n},\quad (a,b) \mapsto y(a,b)
\]
is a universal $U(n)$-invariant map. 
That is, $y$ is $U(n)$-invariant,
i.e.,
for any $g \in U(n)$ and $(a,b) \in \G{m}{n} \times \G{m}{n}$, we have $y(a,b) = y(ga,gb)$, 
and 
for any set $Z$ and any $U(n)$-invariant map $y' : \G{m}{n} \times \G{m}{n} \rightarrow Z$, 
there uniquely exists a map $\varpi : \Range{m}{n} \rightarrow Z$ such that $\varpi \circ y = y'$.
In other words, for $(a,b), (a',b') \in \G{m}{n} \times \G{m}{n}$,
there exists $g \in U(n)$ such that $(a',b') = (ga,gb)$ if and only if $y(a,b) = y(a',b')$.
This fact follows from Roy~\cite[Lemma 1]{Roy2009bca}.

The $i$-th principal angle $y_i(a,b)$ 
is related to
the ``$i$-th angle'' $\theta_i$ between subspaces $a$ and $b$ of
$\mathbb{C}^n$:
\begin{fact}[cf.~Roy~\cite{Roy2009bca}]
Let $\theta_1$ be the smallest angle that occurs between any two unit vectors $a_1 \in a$ and $b_1\in b$, that is, 
\[
\theta_1 =\min \set{ \arccos | (a_1, b_1)|}
{a_1 \in a,\ b_1\in b \text{ with } |a_1| = |b_1| = 1},
\]
where $(a_1,b_1)$ denotes the inner product of $a_1$ and $b_1$ in $\mathbb{C}^n$.
Let $\theta_2$ be the smallest angle that occurs between any two unit vectors $a_2 \in a\cap a_1^\perp$ and
$b_2\in b\cap b_1^\perp$.
Similarly define $\theta_3 ,\ldots ,\theta_m$.
Then for $i=1,2,\ldots ,m$, $y_i(a,b)=\cos ^2 \theta_i$ holds.
\end{fact}
To notice this,
we give easy observations for principal angles as follows:

\begin{obs}
\label{obs:y=111or000}
Let $a,b \in \G{m}{n}$.
Then the following holds:
\begin{enumerate}
\item 
\label{obs:y=111}
$y(a,b) = (1,\dots,1)$ if and only if $a=b$.
\item 
\label{obs:y=000}
$y(a,b) = (0,\dots,0)$ if and only if $a \perp b$ as subspaces of $\mathbb{C}^n$.
\end{enumerate}
\end{obs}

What is a necessary and sufficient condition for
$y(a,b) = (1,\dots,1,0,\dots,0)$?
In fact, the form $y(a,b) = (1,\dots,1,0,\dots,0)$
gives a criterion for when $s_a(b) = b$.
That is, the following proposition holds.

\begin{prop}
\label{prop:equiv_antipodal_set}
The following three conditions on $a,b \in \G{m}{n}$ are equivalent$:$
\begin{enumerate}
\item \label{item:equiv_antipodal_set:antipodal} $s_a(b) = b$.
\item \label{item:equiv_antipodal_set:othogonal} 
$b$ is decomposed as the orthogonal direct sum
$b = (a \cap b) \oplus (a^\perp \cap b)$.
\item \label{item:equiv_antipodal_set:y} $y(a,b)$ is of the form of $(1,\dots,1,0,\dots,0)$.
\end{enumerate}
\end{prop}

\begin{proof}
The equivalence between \eqref{item:equiv_antipodal_set:antipodal} and \eqref{item:equiv_antipodal_set:othogonal} 
was already explained in Section \ref{subsection:antipodal_on_CG}.
If $b = (a \cap b) \oplus (a^\perp \cap b)$,
then the operator $P_a \circ P_b$ on $\mathbb{C}^n$
is the orthogonal projection onto $a \cap b$.
Therefore, we have 
$y(a,b) = (1,\dots,1,0,\dots,0)$,
where the multiplicity of 1 coincides with $\dim a \cap b$.
Hence, we obtain the implication \eqref{item:equiv_antipodal_set:othogonal} $\Rightarrow$ \eqref{item:equiv_antipodal_set:y}.
To completes the proof,
we shall prove \eqref{item:equiv_antipodal_set:y} $\Rightarrow$ \eqref{item:equiv_antipodal_set:antipodal}.
Let us suppose that $y(a,b) = (1,\dots,1,0,\dots,0)$
and put $l$ to the multiplicity of $1$ in $y(a,b)$.
Here we put 
\begin{align*}
a_0 &:= \mathbb{C}\vspan \{e_{1},\dots,e_{m} \}, \\
a_l &:= \mathbb{C}\vspan \{e_{1},\dots,e_{m-l},e_{m+1},\dots,e_{m+l}\},
\end{align*}
where $e_1,\dots,e_n$ is the standard basis of $\mathbb{C}^n$.
Then we can observe that $s_{a_0}(a_l) = a_l$ and $y(a_0,a_l) = y(a,b)$.
By the universality of the $U(n)$-invariant map 
$y : \G{m}{n} \times \G{m}{n} \rightarrow \Range{m}{n}$,
there exits $g \in U(n)$ 
such that $a_0 = ga$ and $a_l = gb$.
Recall that 
$s_{ga}(gb) = g s_a(b)$
since 
$\tilde{s}_{ga} (gv) = g \tilde{s}_a (v)$
for any $v \in \mathbb{C}^n$
(see \eqref{def:tilde-s} for the notation of $\tilde{s}_a$).
Therefore, we have \[
s_{a}(b) = g\inv s_{ga} (gb) = g\inv s_{a_0}(a_l) = g\inv a_l = b.
\]
This completes the proof.
\end{proof}

\subsection{Zonal orthogonal polynomials}
In this subsection, we introduce the definition of 
the zonal orthogonal polynomial $Z_\mu$ for an irreducible representation
$H_\mu$ in $C^0 (\G{m}{n})$,
and we give some properties of $Z_\mu$.
The zonal orthogonal polynomials play a special role
in the Delsarte technique.

An $m$-variables polynomial $p(y_1,y_2,\ldots ,y_m)
\in \C [y_1,y_2,\ldots ,y_m]$ is called
a \emph{symmetric polynomial} if $p$ satisfies
\[
p(y_{\tau (1)},y_{\tau (2)},\ldots ,y_{\tau (m)})
=
p(y_1,y_2,\ldots ,y_m)
\]
for all permutations $\tau$ of $\{1,2,\ldots ,m\}$.
We denote by $\Lambda _m$ the space of all symmetric
polynomials in $y_1,y_2,\ldots ,y_m$.

A $\C$-valued function $f$ on $\G{m}{n}$ is called
a \emph{zonal function} at a point $a\in \G{m}{n}$
if for $b,b'\in \G{m}{n}$, $f(b) =f(b')$
whenever $y(a,b)=y(a,b')$.
Given a symmetric polynomial $p \in \Lambda _m$,
we define the \emph{zonal polynomial $p_a \in C^0(\G{m}{n})$ of $p$ at $a$}
as follows: if $y(a,b)=(y_1 (a,b),y_2 (a,b),\ldots ,y_m (a,b))$
are the principal angles of $a$ and $b$,
then
\[
p_a (b):=p(y_1 (a,b),y_2 (a,b),\ldots ,y_m (a,b)).
\]

Let $H_{\mu}$
be a finite-dimensional irreducible $U(n)$-representation in $C^0 (\G{m}{n})$ 
with the invariant inner product defined by \eqref{def:innerproduct}.
By the Riesz representation theorem,
for each $a\in \G{m}{n}$,
there exists a unique element $Z_{\mu ,a}$ in $H_\mu$
such that for any $f\in H_\mu$,
\begin{equation}
\label{eq:zonal:reproducing_kernel}
\langle f, Z_{\mu ,a}\rangle
=f(a).
\end{equation}
Since the inner product is invariant by $U(n)$, 
the value $Z_{\mu ,a}(b)$
depends only on the $U(n)$-orbit of $(a,b)$
and therefore depends only on the principal angles between $a$ and $b$
(cf.~Roy~\cite{Roy2009bca}).
Thus $Z_{\mu ,a}$ is a zonal fuction at $a$.
In Subsection~\ref{sec:zonal_orth_poly=sum_schur},
we will see that $Z_{\mu ,a}$ is a zonal polynomial of a certain symmetric polynomial.
$Z_{\mu ,a}$ is called the \emph{zonal orthogonal polynomial
at $a$ of $H_\mu$}.
Since $Z_{\mu ,a} (b)$ depends only on the principal angles
betwenn $a$ and $b$,
we sometimes write
$Z_{\mu ,a}(b)=Z_\mu (a,b) =Z_\mu (y(a,b))$.

The zonal orthogonal polynomial $Z_{\mu ,a}$ of $H_\mu$
satisfies
\begin{equation}
\label{eq:normalize_zonal}
Z_\mu (1,1,\ldots ,1) = \dim H_\mu .
\end{equation}
Indeed, \eqref{eq:normalize_zonal} is showed as follows.
Put $N:=\dim H_\mu$, and
fix 
an orthonormal basis $\{\xi _i\}^{N}_{i=1}$
of $H_\mu$ and $a\in \G{m}{n}$.
Since the coefficient of $\xi_i$ in $Z_{\mu ,a}$
is $\langle Z_{\mu ,a}, \xi_i \rangle =\overline{\xi_i (a)}$
by \eqref{eq:zonal:reproducing_kernel},
we have $Z_{\mu ,a} =\sum^N_{i=1}\overline{\xi_i (a)} \xi_i$.
By Observation~\ref{obs:y=111or000} \eqref{obs:y=111} and
\eqref{eq:zonal:reproducing_kernel}, we obtain
\[
Z_\mu (1,1,\ldots ,1)=Z_{\mu ,a} (a)
=\langle Z_{\mu ,a}, Z_{\mu ,a} \rangle
=\langle\sum^{N}_{i=1}\overline{\xi_i (a)} \xi_i,
\sum^{N}_{i=1}\overline{\xi_i (a)} \xi_i \rangle
=\sum^{N}_{i=1} \overline{\xi_i (a)} \xi_i (a) .
\]
Hence it holds that
$\int_{\G{m}{n}}Z_\mu (1,1,\ldots ,1) d\nu
=\int_{\G{m}{n}} \sum^N_{i=1} \overline{\xi_i (a)}\xi_i (a) 
d\nu =
\sum^N_{i=1} \langle \xi_i ,\xi_i\rangle =N$.
Since $\nu$ is a normalized measure,
\eqref{eq:normalize_zonal} follows.

The zonal orthogonal polynomial $Z_{\mu ,a}$ of $H_\mu$
has a certain positivity for a subset of $\G{m}{n}$
as follows:
\begin{fact}[cf.~Roy~{\cite[Lemma 3]{Roy2009bca}}]
\label{fact:positive_Z}
Let $H_{\mu}$
be an irreducible representation in $C^0(\G{m}{n})$
and $Z_\mu$ be the zonal orthogonal polynomial
of $H_{\mu}$.
Then for any subset $X$ of $\G{m}{n}$,
\[
\sum_{a,b\in X} Z_\mu (y(a,b)) \ge 0.
\]
Equality holds if and only if
\[
\sum_{a\in X} Z_{\mu , a}= 0.
\]
\end{fact}

The following proposition gives equivalent conditions of
the definition of designs on $\G{m}{n}$:
\begin{prop}\label{prop:eq_design_Z}
Let $X$ be a non-empty finite subset of $\G{m}{n}$ and $\mathcal{T}$ a finite subset of $\Partition{m}$.
Then the following conditions on $(X,\mathcal{T})$ are equivalent$:$
\begin{enumerate}
\item \label{item:eq_design:design}
$X$ is a $\mathcal{T}$-design on $\G{m}{n}$.
\item \label{item:eq_design:zero}
$\sum_{a \in X} f(a) = 0$ for any $\mu \in \mathcal{T} \setminus \{(0)\}$ and $f \in H_\mu$.
\item \label{item:eq_design:Z} $\sum_{a,b \in X} Z_{\mu}(y(a,b)) = 0$ for any $\mu \in \mathcal{T} \setminus \{(0)\}$.
\end{enumerate}
\end{prop}

\begin{proof}
The proof parallels to that of \cite[Lemma 7]{Roy2009bca}.
\end{proof}

\subsection{An expression of the zonal orthogonal polynomials using the Schur polynomials}
\label{sec:zonal_orth_poly=sum_schur}

In this subsection, we give an expression of the zonal orthogonal
polynomial $Z_\mu$ for an irreducible representation
$H_\mu$ in $C^0 (\G{m}{n})$
by using the Schur polynomials.

For $i=0,1,\ldots ,m$, let
\[
\fre _i (y_1,y_2,\ldots ,y_m) =
\sum_{1\le k_1 < k_2 <\cdots < k_i \le m}
y_{k_1}y_{k_2}\cdots y_{k_i}
\in \Lambda_m
\]
and for $i\in \Z_{> 0}$, let
\[
\frh _i (y_1,y_2,\ldots ,y_m) =
\sum_{1\le k_1 \le k_2 \le \cdots \le k_i \le m}
y_{k_1}y_{k_2}\cdots y_{k_i}
\in \Lambda_m.
\]
The polynomials $\fre _i$ and $\frh _i $
are called the \emph{$i$-th symmetric polynomial}
and the \emph{$i$-th complete symmetric polynomial} in
$y_1,y_2,\ldots ,y_m$, respectively.
It is well known that $\Lambda_m=\C [\fre_1,\fre_1,\ldots ,\fre_m]$.

If $y=(y_1,y_2,\ldots ,y_m)$ are variables and
$\mu =(\mu_1,\mu_2,\ldots ,\mu_m)$
is in $\Partition{m}$,
then the (unnormalized) \emph{Schur polynomial for $\mu$}
is defined as
\[
X_\mu (y):=
\frac{\det (y_i^{\mu_j +m-j})^m_{i,j=1}}
{\det (y_i^{m-j})^m_{i,j=1}}.
\]
Each Schur polynomial $X_\mu$ is in $\Lambda _m$.
\emph{The normalized Schur polynomial $X^\ast_{\mu}$}
is the multiple of $X_\mu$ such that
$X^\ast_{\mu}(1,1,\ldots ,1)=1$, i.e.,
\[
X^\ast_{\mu} = \frac{1}{X_{\mu}(1,1,\ldots ,1)}X_\mu.
\]

A partition $\mu\in \Partition{m}$
can be viewed as a Ferrers shape obtained by
placing cells in $m$ left-justified rows with $\mu_i$
boxes in row $i$.
For example, if $\mu=(2,1,1,0,\ldots ,0)$ then its shape

\hspace{-1.5cm}
\ifx\JPicScale\undefined\def\JPicScale{1}\fi
\unitlength \JPicScale mm
\begin{picture}(80,25)(0,0)
\put(55,20){\makebox(0,0)[cc]{$\mu=$}}

\put(80,20){\makebox(0,0)[cc]{.}}

\linethickness{0.3mm}
\put(65,25){\line(1,0){10}}
\put(65,20){\line(0,1){5}}
\put(75,20){\line(0,1){5}}
\put(65,20){\line(1,0){10}}
\linethickness{0.3mm}
\put(65,10){\line(1,0){5}}
\put(65,10){\line(0,1){15}}
\put(70,10){\line(0,1){15}}
\put(65,25){\line(1,0){5}}
\linethickness{0.3mm}
\put(65,15){\line(1,0){5}}
\end{picture}

\vspace{-1cm}
For $\mu\in \Partition{m}$,
let $\mu'$ denote the partition conjugate
to $\mu$ whose parts are the column lengths of the Ferrers shape
of $\mu$.
In the preceding example,
$\mu'=(3,1,0,\ldots ,0)$
and its shape

\ifx\JPicScale\undefined\def\JPicScale{1}\fi
\unitlength \JPicScale mm
\begin{picture}(70,20)(0,0)
\linethickness{0.3mm}
\put(50,20){\line(1,0){5}}
\put(50,10){\line(0,1){10}}
\put(55,10){\line(0,1){10}}
\put(50,10){\line(1,0){5}}
\linethickness{0.3mm}
\put(50,20){\line(1,0){15}}
\put(50,15){\line(0,1){5}}
\put(65,15){\line(0,1){5}}
\put(50,15){\line(1,0){15}}
\linethickness{0.3mm}
\put(60,15){\line(0,1){5}}
\put(40,15){\makebox(0,0)[cc]{$\mu '=$}}

\put(70,15){\makebox(0,0)[cc]{.}}

\end{picture}

\vspace{-1cm}
The following theorem yields that
the Schur polynomials can be expressed by using 
the symmetric polynomials or 
the complete symmetric polynomials.
The details can be referred to Macdonald~\cite{Macdonald1995sfa}.
\begin{fact}[Jacobi-Trudi identity and Giambelli identity]
\label{fact:schur<->symmetric_poly}
\[
X_\mu=\det (\frh_{\mu_i-i+j})^{\len (\mu)}_{i,j=1}\ 
\text{and}\ 
 X_{\mu }=\det (\fre_{{\mu '}_i-i+j})^{\len (\mu ')}_{i,j=1},
\]
where $\len (\mu)$ is the number of the non-zero entries of
$\mu=(\mu_1,\mu_2,\ldots )$.
In particular, we have
\[
X_{(i)}=\frh_i\ 
\text{and}\ 
 X_{(1^i)}=\fre_i.
\]
\end{fact}

To describe the zonal orthogonal polynomials for
$H_\mu$,
first define the \emph{ascending product}
\[
(c)_s:=\prod^s_{i=1}(c+i-1),
\]
with initial value $(c)_0:=1$,
and given a partition
$\sigma =(\sigma_1,\ldots ,\sigma_m)$,
define \emph{complex hypergeometric coefficients}
\[
[c]_\sigma
:=\prod^m_{i=1}
(c-i+1)_{\sigma_i}.
\]
Further assume we have a partial order $\le$
on $\Partition{m}$ defined such that
$(\sigma_1,\ldots ,\sigma_m) \le (\kappa_1,\ldots ,\kappa_m)$
if and only if $\sigma_i \le \kappa_i$
for all $i$.
Let $y+1:=(y_1+1,y_2+1,\ldots ,y_m+1)$.
The \emph{complex hypergeometric binomial coefficients}
$\qbinom{\kappa}{\sigma}$
are given by the formula
\[
X^\ast_\kappa (y+1)=
\sum_{\sigma \le \kappa} \qbinom{\kappa}{\sigma} X^\ast _\sigma (y).
\]
If there exists no order between $\kappa$ and $\sigma$,
then we define $\qbinom{\kappa}{\sigma}:=0$.
For $\sigma ,\kappa \in \Partition{m}$,
let $\rho_\sigma :=\sum ^m_{i=1}\sigma_i (\sigma_i-2i+1)$
and $s=\sum^m_{i=1} \sigma_i$ and $k=\sum^m_{i=1} \kappa_i$.
Also let
\[
[c]_{(\kappa,\sigma)}
:=
\sum_{i\in M}
\frac{\qbinom{\kappa}{\sigma^{\uparrow i}}
\qbinom{\sigma^{\uparrow i}}{\sigma}}
{(k-s)\qbinom{\kappa}{\sigma}}
\frac{[c]_{(\kappa,\sigma^{\uparrow i})}}
{c+\frac{\rho _\kappa -\rho_\sigma}{k-s}},
\]
where
\[
M:=\left\{i=1,2,\ldots, m
\mid
\substack{
\text{
$\sigma^{\uparrow i}
:=(\sigma_1,\ldots ,\sigma_{i-1},\sigma_i+1,\sigma_{i+1},
\ldots ,\sigma_m)$
is non-increasing}\\
\text{and $\sigma^{\uparrow i} \le \kappa$}}
\right\}
\]

We can now define the zonal orthogonal polynomials for $H_\mu$.
The following result
is due to James--Constantine~\cite{James1974gjp}.

\begin{fact}
Up to normalization, the zonal orthogonal polynomial $\tilde{Z}_\mu$
of $H_\mu$
is
\[
\tilde{Z}_\mu (y):=\sum_{\sigma \le \mu}
\frac{(-1)^s\qbinom{\mu}{\sigma}[n]_{(\mu,\sigma)}}
{[m]_{\sigma}}
X^\ast_\sigma (y),
\]
where
$y=(y_1,y_2,\ldots ,y_m)\in \Range{m}{n}$.
\end{fact}
We note that, although $\tilde{Z}_{\mu}$
does not satisfy \eqref{eq:normalize_zonal}.
By multiplying $\tilde{Z}_{\mu}$ by
$\dim H_\mu/\tilde{Z}_{\mu}(1,1,\ldots ,1)$, we obtain
the ``normalized'' zonal orthogonal polynomial $Z_\mu$.

\subsection{Some formulas for zonal polynomials}
For an integer $k$ and a non-negative integer $r$,
the binomial coefficient $\binom{k}{r}$
is defined by 
$\binom{k}{r}=\prod^{r-1}_{i=0}(k-i)/(r-i)$
if $r>0$ and 1 if $r=0$.
The binomial coefficient satisfies
$\binom{-k}{r}=(-1)^r\binom{k+r-1}{r}$
and $\binom{k}{r}=0$ if $k<r$.

\begin{lem}
\label{lem:binom_formula}
The binomial coefficients satisfy the following relations:
\begin{enumerate}
\item
\label{item:bino1}
$\binom{n-k}{m-k}\binom{n}{k}=\binom{n}{m}\binom{m}{k}$,
\item
\label{item:bino2}
$\sum^m_{k=0}(-1)^{k}\binom{p}{k}\binom{n-k}{m-k}=\binom{n-p}{m}$,
\item
\label{item:bino3}
$($Vandermonde identity$)$
$\sum^p_{k=0}\binom{n}{p-k}\binom{m}{k}=\binom{n+m}{p}$,
\item
\label{item:bino4}
$($cf.~Delsarte \cite{Delsarte1973aat}$)$
$\sum^r_{t=i}(-1)^{t-i} \binom{t}{i}\binom{n-t}{r-t}\binom{u}{t}
=
\binom{n-u}{r-i}\binom{u}{i}$.
\end{enumerate}
\end{lem}
Note that Lemma~\ref{lem:binom_formula} \eqref{item:bino4}
when $t=0$
leads to Lemma~\ref{lem:binom_formula} \eqref{item:bino2}.
Still, for convenience, we use both formulas as the situation demands.

\begin{prop}
\label{prop:Z=sumX}
For each $i=0,1,\ldots ,m$,
the zonal orthogonal polynomial for $H_{(1^i)}$ is
\[
Z_{(1^i)}
=
\frac{(n-2i+1)\binom{n+1}{i}^2}{(n+1) \binom{n-m}{i}}
\sum^i_{j=0}
(-1)^{i-j}
\binom{n-i+1}{j}
\binom{m-j}{i-j}
X^\ast_{
(1^j)}.
\]
\end{prop}

\begin{proof}
Firstly, we show
\begin{equation}
\label{eq:q-binom_i_j}
\qbinom{(1^{i})}{(1^{j})}
=\binom{i}{j}.
\end{equation}
By Fact~\ref{fact:schur<->symmetric_poly}, we have
\begin{eqnarray*}
X^\ast_{(1^{i})}(y+1)
&=&
\frac{1}{\binom{m}{i}}\fre_i(y+1)\\
&=&
\frac{1}{\binom{m}{i}}
\sum^i_{j=0}\binom{m-i}{m-j}\fre_j(y)\\
&=&
\sum^{i}_{j=0}
\frac{1}{\binom{m}{i}}
\binom{m-j}{i-j}\binom{m}{j}
X^\ast_{(1^{j})}(y)\\
&=&
\sum^{i}_{j=0}
\binom{i}{j}
X^\ast_{(1^{j})}(y).
\end{eqnarray*}
In the last line, we use Lemma~\ref{lem:binom_formula} \eqref{item:bino1}.
This implies \eqref{eq:q-binom_i_j}.

We second show that for each
$i\in \mathbb{Z}_{> 0}$ and
$j=0,1,\ldots ,i$,
\[
[c]_{((1^{i}),(1^{j}))}=
\prod^i_{l=j}\frac{1}{c-i-l+1}
\]
holds.
Note that $[c]_{((1^{i}),(1^{i}))}$ is an indeterminate
and let $[c]_{((1^{i}),(1^{i}))}=\frac{1}{c-2i+1}$.
The sequence $(1^{j})^{\uparrow k}$ of integers is non-increasing only if $k=1,j+1$.
Moreover 
there is no order between
$(1^{i})$ and $(1^{j})^{\uparrow 1}=(2,1^{j-1})$.
This implies
$\qbinom{(1^{i})}{(1^{j})^{\uparrow 1}}=0$.
Hence we have
\begin{eqnarray*}
[c]_{((1^{i}),(1^{j}))}
&=&
\frac{\qbinom{(1^{i})}{(1^{j})^{\uparrow j+1}}
\qbinom{(1^{j})^{\uparrow j+1}}
{(1^{j})}}
{(i-j)\qbinom{(1^{i})}{(1^{j})}}
\frac{[c]_{((1^{i}),
(1^{j})^{\uparrow j+1})}}{c+\frac{\rho_{(1^{i})}
-\rho_{(1^{j})}}{i-j}}\\
&=&
\frac{1}{c-i-j+1}
[c]_{((1^{i}),
(1^{j+1}))}.
\end{eqnarray*}
In this calculation, we use
$\rho_{(1^{i})}
=\sum^i_{k=1}(1-2k+1)
=i-i^2$.
By induction in $j$, the desired result follows.

Since the complex hypergeometric coefficients are
$[a]_{(1^i)}=\prod^i_{k=1}(a-k+1)$,
we obtain
\[
\tilde{Z}_{(1^{i})}
=
\sum^i_{j=0}
\frac{
(-1)^j \binom{i}{j}
\prod^i_{l=j}\frac{1}{n-i-l+1}
}{\prod^j_{k=1}
(m-k+1)}
X^\ast_{(1^{j})}.
\]
Finally, we have to normalize $\tilde{Z}_{(1^{i})}$.
Let
\[
\dot{Z}_{(1^{i})}:=(-1)^i \prod^i_{l=0}(n-i-l+1) \prod^i_{k=1}(m-k+1) 
\cdot
\tilde{Z}_{(1^{i})}.
\]
Then by
\begin{eqnarray*}
\lefteqn{
(-1)^i \prod^i_{l=0}(n-i-l+1) \prod^i_{k=1}(m-k+1)
\times
\frac{
(-1)^j \binom{i}{j}
\prod^i_{l=j}\frac{1}{n-i-l+1}
}{\prod^j_{k=1}
(m-k+1)}
}\\
&=&
(-1)^{i-j}\binom{i}{j}
\prod^{j-1}_{l=0}(n-i-l+1)
\prod^{i}_{k=j+1}(m-k+1)\\
&=&
(-1)^{i-j}\frac{i!}{j! (i-j)!}
\prod^{j-1}_{l=0}(n-i-l+1)
\prod^{i}_{k=j+1}(m-k+1)\\
&=&
(-1)^{i-j}
i!\binom{n-i+1}{j}
\binom{m-j}{i-j},
\end{eqnarray*}
we obtain $\dot{Z}_{(1^{i})}=\sum^i_{j=0}
(-1)^{i-j}
i!\binom{n-i+1}{j}
\binom{m-j}{i-j}
X^\ast_{(1^{j})}$.
Since $\dot{Z}_{(1^{i})}(1,1,\ldots ,1)$ coincides with
\[
\sum^i_{j=0}
(-1)^{i-j}
i!\binom{n-i+1}{j}
\binom{m-j}{i-j}=
(-1)^i i! \binom{m-n+i-1}{i}=
i! \binom{n-m}{i}
\]
by Lemma~\ref{lem:binom_formula} \eqref{item:bino2}
and $\dim H_{(1^{i})}$
is equals to $\frac{n-2i+1}{n+1}\binom{n+1}{i}^2$
by \eqref{eq:dimH_(1^i)},
the normalized zonal orthogonal polynomial for $H_{(1^i)}$
is
\begin{eqnarray*}
Z_{(1^{i})}
&=&
\frac{\dim H_{(1^{i})}}{\dot{Z}_{(1^{i})}(1,1,\ldots ,1)}\dot{Z}_{(1^{i})}\\
&=&
\frac{(n-2i+1)\binom{n+1}{i}^2}{(n+1) \binom{n-m}{i}}
\sum^i_{j=0}
(-1)^{i-j}
\binom{n-i+1}{j}
\binom{m-j}{i-j}
X^\ast_{
(1^j)}.
\end{eqnarray*}
\end{proof}

\begin{rem}
\label{rem:Z_i}
We can also calculate the expression of
the orthogonal polynomial $Z_{(i)}$ of $H_{(i)}$ by
using $X^\ast_{(j)}$'s as follows:
\[
Z_{(i)}
=
\frac{(n+2i-1)\binom{n+i-2}{i}^2}{(n-1)\binom{n-m+i-1}{i}}
\sum^i_{j=0}
(-1)^{i-j}
\binom{n+i+j-2}{j}
\binom{m+i-1}{i-j}
X^\ast_{(j)}.
\]
\end{rem}

\begin{prop}
\label{prop:X=sumZ}
The normalized Schur polynomial $X^\ast_{(1^i)}$ can be represented
by using the zonal orthogonal polynomial $Z_{(1^j)}$:
\begin{equation}
\label{eq:X=sumZ}
X^\ast_{(1^i)}=\sum^i_{j=0}\frac{n+1}{n-j+1}
\frac{\binom{m-j}{i-j}\binom{n-m}{j}}
{\binom{n-j}{i}\binom{n+1}{j}^2}Z_{(1^j)}.
\end{equation}
\end{prop}
\begin{proof}
The validity can be verified
to check that the product
of two matrices
\[
\left(
(-1)^{i-j}
\frac{(n-2i+1)\binom{n+1}{i}^2
\binom{n-i+1}{j}
\binom{m-j}{i-j}
}{(n+1) \binom{n-m}{i}}
\right)^m_{i,j=0}
\text{and}\ 
\left(
\frac{(n+1)\binom{m-j}{i-j}\binom{n-m}{j}}
{(n-j+1)\binom{n-j}{i}\binom{n+1}{j}^2}
\right)^m_{i,j=0}
\]
obtained from the coefficients of $X^\ast_{(1^j)}$'s
in $Z_{(1^i)}$ and
the coefficients of $Z_{(1^j)}$'s
in $X^\ast_{(1^i)}$,
respectively, 
is the identity matrix.
The $(i,k)$-entry of the product of these matrices
is calculated as
\begin{equation}
\label{eq:(i,k)-entry}
\frac{(n-2i+1)\binom{n-m}{k}\binom{m}{i}\binom{n+1}{i}^2}
{(n-k+1)\binom{n-m}{i}\binom{m}{k}\binom{n+1}{k}^2}
\sum^m_{j=0}
(-1)^{i-j}
\binom{i}{j}\binom{j}{k}\frac{\binom{n-i+1}{j}}{\binom{n-k}{j}}
\end{equation}
by Lemma~\ref{lem:binom_formula} \eqref{item:bino1}.
If $k>i$, then \eqref{eq:(i,k)-entry} vanishes
by $\binom{i}{j}\binom{j}{k}=0$ for each $j=0,1,\ldots ,m$.
If $k=i$, then the index $j$ is restricted to $i$.
Hence \eqref{eq:(i,k)-entry} is equal to 1.
If $k<i$, then 
using Lemma~\ref{lem:binom_formula} \eqref{item:bino1} and \eqref{item:bino4}, 
we obtain that
\eqref{eq:(i,k)-entry} is equal to
\[
\frac{(n-2i+1)\binom{n-m}{k}\binom{m}{i}\binom{n+1}{i}^2}
{(n-k+1)\binom{n-m}{i}\binom{m}{k}\binom{n+1}{k}^2}
\frac{\binom{(n-k)-i}{(n-i+1)-k} \binom{i}{k}}{\binom{n-k}{n-i+1}}
\]
and, by $\binom{(n-k)-i}{(n-i+1)-k}=0$, the above value vanishes.
\end{proof}

\begin{prop}
\label{prop:zonal_poly_(2111)}
\begin{eqnarray*}
Z_{(2,1^{i-1})}
&=&
f_{2}\sum^i_{j=1}(-1)^{i-j}\frac{1}{j+1}
\binom{m-j}{i-j}\binom{n-i}{j-1}
X^\ast_{(2,1^{j-1})}\\
&&
+f_{1}\sum^i_{j=1}
(-1)^{i-j+1}\frac{1}{j}
\binom{m-j}{i-j}\binom{n-i}{j-1}X^\ast_{(1^{j})}
+f_{0}X^\ast_{(0)}
\end{eqnarray*}
with
\begin{eqnarray*}
f_{2}
&=&
\frac{i(i+1)(n+2)(n+3)(n-2i+1)\binom{n+1}{i+1}^2}
{(n-i+2)(n-m+1)\binom{n-m}{i}},
\\
f_{1}
&=&
\frac{i(i+1)(m+1)(n+3)(n-2i+1)\binom{n+1}{i+1}^2}
{(n-i+2)(n-m+1)\binom{n-m}{i}}
\ \text{and}
\\
f_{0}
&=&
(-1)^{i+1}
\frac{i^2 (m+1)(n+3)(n-2i+1)\binom{n+1}{i+1}^2 \binom{m}{i}}
{(n-i+2)^2 (n-m+1)\binom{n-m}{i}}.
\end{eqnarray*}
\end{prop}

\begin{proof}
The conjugate partition of $(2,1^{i-1})\in \Partition{m}$
is $(i,1)$ and by the Giambelli identity in
Fact~\ref{fact:schur<->symmetric_poly}, we have
\[
X_{(2,1^{i-1})}
=
X_{(i,1)'}
=
   \left|
   \begin{array}{cc}
      \fre_{i+1-1} & \fre_{i+2-1} \\
      \fre_{1+1-2} & \fre_{1+2-2} \\
   \end{array}
   \right|
=\fre_i \fre_1 -\fre_{i+1}.
\]
Then the normalized Schur polynomial
is $X^\ast_{(2,1^{i-1})}=\frac{1}{i \binom{m+1}{i+1}}(\fre_i \fre_1 -\fre_{i+1})$.
By the definition of the complex hypergeometric binomial coefficients,
we can check
\[
\qbinom{(2,1^{i-1})}
{(2,1^{j-1})}
=
\frac{i+1}{j+1}
\binom{i-1}{j-1},\ 
\qbinom{(2,1^{i-1})}
{(1^{j})}
=
\frac{i+1}{i}\binom{i}{j},\ 
\qbinom{(2,1^{i-1})}
{(0)}
=
1.
\]
Let $[c]_{((2,1^{i-1}),(2,1^{i-1}))}:=1$.
Then by the definition of $[c]_{(\kappa ,\sigma)}$,
we can check
\[
[c]_{((2,1^{i-1}),(2,1^{j-1}))}=\prod^{i-1}_{k=j}\frac{1}{c-i-k+1}.
\]
Next by using a proof of induction,
we prove that for $j=1,2,\ldots ,i$,
\[
[c]_{((2,1^{i-1}),(1^{j}))}=\frac{1}{c+2}\prod^{i-1}_{k=j}
\frac{1}{c-i-k+1}
\]
holds.
When $j=i$, we can check $[c]_{((2,1^{i-1}),(1^i))}=1/(c+2)$.
For each $j=1,2,\ldots , i-1$, assume that
$[c]_{(2,1^{i-1}),(1^{j+1})}=\frac{1}{c+2}\prod^{i-1}_{k=j+1}\frac{1}{c-i-k+1}$ holds.
Then we have
\begin{eqnarray*}
\lefteqn{
[c]_{(2,1^{i-1}),(1^j)}}\\
&=&
\frac{1}{((i+1)-j)\qbinom{(2,1^{i-1})}{(1^{j})}}
\frac{1}{c+\frac{\rho_{(2,1^{i-1})}-\rho_{(1^{j})}}{((i+1)-j)}}\\
&&
\times
\left(
\qbinom{(2,1^{i-1})}{(2,1^{i-1})}\qbinom{(2,1^{j-1})}{(1^{j})}
[c]_{((2,1^{i-1}),(2,1^{j-1}))}
+
\qbinom{(2,1^{i-1})}{(1^{j+1})}\qbinom{(1^{j+1})}{(1^{j})}
[c]_{((2,1^{i-1}),(1^{j+1}))}
\right)\\
&=&
\frac{1}{c+2}
\prod^{i-1}_{k=j}\frac{1}{c-i-k+1}.
\end{eqnarray*}
Hence the desired result holds.
Finally we can also check \[
[c]_{((2,1^{i-1}),(0))}
=
\frac{1}{c-(i-2)}
[c]_{((2,1^{i-1}),(1))}=
\frac{1}{c-i+2}
\frac{1}{c+2}
\prod^{i-1}_{k=1}\frac{1}{c-i-k+1}.\]
Thus up to normalization, the zonal orthogonal polynomial for
$H_{(2,1^{i-1})}$
is written in
\begin{eqnarray*}
\tilde{Z}_{(2,1^{i-1})}
&=&
\sum^i_{j=1}
(-1)^{j+1}\frac{(i+1)! (n+2)(n-i+2)}{i (j+1)}
\binom{m-j}{i-j}\binom{n-i}{j-1}
X^\ast_{(2,1^{j-1})}\\
&&
+
\sum^i_{j=1}
(-1)^{j}\frac{(i+1)! (m+1)(n-i+2)}{i j}
\binom{m-j}{i-j}\binom{n-i}{j-1}
X^\ast_{(1^{j})}\\
&&
+(m+1) i! \binom{m}{i} X^\ast_{(0)}.
\end{eqnarray*}
Therefore,
by $\tilde{Z}_{(2,1^{i-1})}(1,1,\ldots ,1)=
(-1)^{i+1} i! (n-m+1) \binom{n-m}{i}$
and \eqref{eq:dimH_(2,1^i-1)},
the normalized zonal orthogonal polynomial
$Z_{(2,1^{i-1})}$ is given as desired.
\end{proof}

\begin{lem}
\label{lem:Z(1)Z(1^i)}
For any $i = 1,2,\dots,m$, the product $Z_{(1)} \cdot Z_{(1^i)}$ can be written by 
\[
Z_{(1)} \cdot Z_{(1^i)} =
a_i Z_{(2,1^{i-1})}
+ b^{(i)}_{i+1}Z_{(1^{i+1})}
+ b^{(i)}_{i} Z_{(1^i)} + b^{(i)}_{i-1} Z_{(1^{i-1})}
\]
with
\begin{eqnarray*}
a_i&=&
\frac{(i+1)(m+1)n(n-1)(n-i+2)(n-m+1)}{i m (n+2)(n+3)(n-i+1)(n-m)}
>0,\\
b^{(i)}_{i+1}&=&
\frac{(i+1)(m-i)n(n-1)(n+1)(n-m-i)}{m (n-i+1)(n-2i)(n-2i-1)(n-m)}
\ge 0,\\
b^{(i)}_{i}&=&
\frac{2 i (n-1)(n+1)(n-i+1)(n-2m)^2}{m (n+2) (n-2i) (n-2i+2)(n-m)}
\ge 0\ \text{and}\\
b^{(i)}_{i-1}&=&
\frac{(m-i+1) n(n+1)(n-1) (n-i+2)(n-m-i+1)}{i m (n-2i+2)(n-2i+3) (n-m)}
> 0.
\end{eqnarray*}
\end{lem}

\begin{rem}
The positivities of coefficients in the formula in Lemma \ref{lem:Z(1)Z(1^i)} 
can be explained in terms of branching rules of 
the $U(n)$-representation $V_{\phi((1))} \otimes V_{\phi((1^i))}$,
see Section \ref{subsection:Harmonic} for the notation of $V_{\phi(\mu)}$.
We omit the details here.
\end{rem}

\begin{proof}
We prove this lemma by direct calculation.
By Proposition~\ref{prop:Z=sumX}, we have
\begin{eqnarray}
\notag
\lefteqn{
Z_{(1)}\cdot Z_{(1^i)}
}\\
\notag
&=&
\frac{(n-1)(n-2i+1)\binom{n+1}{i}^2}{(n-m)\binom{n-m}{i}}
\left(
n\sum^i_{j=0}
(-1)^{i-j}
\binom{n-i+1}{j}
\binom{m-j}{i-j}
X^\ast_{(1)}
X^\ast_{
(1^j)}
\right.\\
&&
\label{eq:Z1Z1^j_first}
\left.
-m
\sum^i_{j=0}
(-1)^{i-j}
\binom{n-i+1}{j}
\binom{m-j}{i-j}
X^\ast_{
(1^j)}
\right).
\end{eqnarray}
By the definition of $X^\ast_{(2,1^{j-1})}$,
for $j=1,2,\ldots ,m$,
we have
\[
X^\ast_{(2,1^{j-1})}
=\frac{1}{j \binom{m+1}{j+1}}
(\fre_1 \fre_{j}-\fre_{j+1})
=\frac{1}{j \binom{m+1}{j+1}}
\left(
m \binom{m}{j}X^\ast_{(1)} X^\ast_{(1^j)}
-\binom{m}{j+1}X^\ast_{(1^{j+1})}
\right).
\]
This implies
$X^\ast_{(1)} X^\ast_{(1^j)}=
\frac{j(m+1)}{(j+1)m}X^\ast_{(2,1^{j-1})}
+\frac{m-j}{(j+1)m}X^\ast_{(1^{j+1})}$.
Then we obtain
\begin{eqnarray}
\notag
\lefteqn{
\sum^i_{j=0}(-1)^{i-j}
\tbinom{n-i+1}{j}\tbinom{m-j}{i-j}
X^\ast_{(1)} X^\ast_{(1^j)}
}\\
&=&
\notag
\sum^i_{j=1}(-1)^{i-j}
\tbinom{n-i+1}{j}\tbinom{m-j}{i-j}
\left(
\tfrac{j(m+1)}{(j+1)m}X^\ast_{(2,1^{j-1})}
+\tfrac{m-j}{(j+1)m}X^\ast_{(1^{j+1})}
\right)
+(-1)^i\tbinom{m}{i}X^\ast_{(1)}\\
&=&
\notag
\tfrac{(n-i+1)(m+1)}{m}\sum^i_{j=1}(-1)^{i-j}
\tfrac{1}{j+1}\tbinom{n-i}{j-1}\tbinom{m-j}{i-j}
X^\ast_{(2,1^{j-1})}
+\tbinom{n-i+1}{i}\tfrac{m-i}{(i+1)m}X^\ast_{(1^{i+1})}\\
&&
\label{eq:X1X1^j_expand}
-\sum^i_{j=1}(-1)^{i-j}\tfrac{m-j+1}{j m}
\tbinom{n-i+1}{j-1}\tbinom{m-j+1}{i-j+1}
X^\ast_{(1^{j})}.
\end{eqnarray}
On the other hand
we have,
by Proposition~\ref{prop:zonal_poly_(2111)},
\begin{eqnarray}
\notag
\lefteqn{
\sum^i_{j=1}(-1)^{i-j}\tfrac{1}{j+1}
\tbinom{n-i}{j-1}\tbinom{m-j}{i-j}X^\ast_{(2,1^{j-1})}
}\\
&=&
\notag
\tfrac{(n-m+1)(n-i+2)\binom{n-m}{i}}
{i (i+1) (n+2) (n+3) (n-2i+1)\binom{n+1}{i+1}^2}
Z_{(2,1^{i-1})}
+\tfrac{m+1}{n+2}
\sum^i_{j=1}(-1)^{i-j}\tfrac{1}{j}
\tbinom{n-i}{j-1}\tbinom{m-j}{i-j}X^\ast_{(1^{j})}\\
&&
\label{eq:X21^j-1_expand}
+(-1)^i \tfrac{i(m+1)}{(i+1)(n+2)(n-i+2)} \tbinom{m}{i}X^\ast_{(0)}
\end{eqnarray}
and by Proposition~\ref{prop:Z=sumX},
\begin{equation}
\label{eq:X1^i+1_expand}
X^\ast_{(1^{i+1})}
=
\tfrac{(n+1)\binom{n-m}{i+1}}
{(n-2i-1)\binom{n+1}{i+1}^2\binom{n-i}{i+1}}
Z_{(1^{i+1})}
+\tfrac{1}{\binom{n-i}{i+1}}
\sum^{i}_{j=0}(-1)^{i-j}
\tbinom{n-i}{j}\tbinom{m-j}{i+1-j}
X^\ast_{(1^j)}.
\end{equation}
Applying 
\eqref{eq:X1X1^j_expand},
\eqref{eq:X21^j-1_expand}
and
\eqref{eq:X1^i+1_expand}
to
\eqref{eq:Z1Z1^j_first},
we have
\begin{eqnarray*}
\lefteqn{
Z_{(1)}\cdot Z_{(1^i)}
}\\
&=&
\tfrac{(i+1)(m+1)n(n-1)(n-i+2)(n-m+1)}{i m (n+2)(n+3)(n-i+1)(n-m)}
Z_{(2,1^{i-1})}
+\tfrac{(i+1)(m-i)n(n-1)(n+1)(n-m-i)}
{m (n-i+1)(n-2i)(n-2i-1)(n-m)}Z_{(1^{i+1})}\\
&&
+\tfrac{(n-1)(n-2i+1)\binom{n+1}{i}^2}{(n-m)\binom{n-m}{i}}
\left(
\tfrac{(m+1)^2 n (n-i+1)}{m (n+2)}
\sum^i_{j=1}(-1)^{i-j}\tfrac{1}{j}
\tbinom{n-i}{j-1}\tbinom{m-j}{i-j}X^\ast_{(1^{j})}
\right.\\
&&
+(-1)^i \tfrac{i(m+1)^2 n (n-i+1)}{(i+1)m(n+2)(n-i+2)} \tbinom{m}{i}X^\ast_{(0)}
+\tfrac{(m-i) n (n-i+1)}{m (n-2i+1)(n-2i)}\sum^{i}_{j=0}
(-1)^{i-j}\tbinom{n-i}{j}\tbinom{m-j}{i+1-j}X^\ast_{(1^j)}\\
&&
\left.
-\sum^i_{j=1}(-1)^{i-j}
\tfrac{n(m-j+1)}{j m}\tbinom{n-i+1}{j-1}\tbinom{m-j+1}{i-j+1}X^\ast_{(1^{j})}
-m\sum^i_{j=0}(-1)^{i-j}
\tbinom{n-i+1}{j}\tbinom{m-j}{i-j}X^\ast_{(1^j)}
\right).
\end{eqnarray*}
Since the coefficient of $X^\ast_{(1^i)}$
in the third term of the above equation is
\[
\frac{(n-1)(n-2i+1)\binom{n+1}{i}^2}{(n-m)\binom{n-m}{i}}
\frac{2 i (n-i+1) (n-2m)^2}{m (n+2) (n-2i) (n-2i+2)}\binom{n-i+1}{i},
\]
the coefficient of $Z_{(1^i)}$ in $Z_{(1)}\cdot Z_{(1^i)}$
is determined as
\[
\frac{2 i (n-1)(n+1)(n-i+1)(n-2m)^2}{m (n+2) (n-2i) (n-2i+2)(n-m)}.
\]
The remainder term of $Z_{(1)}\cdot Z_{(1^i)}$
can be written in the multiple of $Z_{(1^{i-1})}$
as follows:
\begin{eqnarray*}
\lefteqn{
\tfrac{(n-1)(n-2i+1)\binom{n+1}{i}^2}{(n-m)\binom{n-m}{i}}
\left(
\tfrac{(m+1)^2 n (n-i+1)}{m (n+2)}
\sum^{i-1}_{j=1}(-1)^{i-j}\tfrac{1}{j}
\tbinom{n-i}{j-1}\tbinom{m-j}{i-j}X^\ast_{(1^{j})}
\right.
}\\
&&
+(-1)^i \tfrac{i(m+1)^2 n (n-i+1)}{(i+1)m(n+2)(n-i+2)} \tbinom{m}{i}X^\ast_{(0)}
+\tfrac{(m-i) n (n-i+1)}{m (n-2i+1)(n-2i)}
\sum^{i-1}_{j=0}
(-1)^{i-j}
\tbinom{n-i}{j}
\tbinom{m-j}{i+1-j}
X^\ast_{(1^j)}\\
&&
-\sum^{i-1}_{j=1}
(-1)^{i-j}
\tfrac{n(m-j+1)}{j m}
\tbinom{n-i+1}{j-1}
\tbinom{m-j+1}{i-j+1}
X^\ast_{(1^{j})}
-m\sum^{i-1}_{j=0}(-1)^{i-j}
\tbinom{n-i+1}{j}
\tbinom{m-j}{i-j}
X^\ast_{(1^j)}\\
&&
\left.
-\tfrac{2i (n-i+1) (n-2m)^2}{m (n+2) (n-2i) (n-2i+2)}
\sum^{i-1}_{j=0}(-1)^{i-j}
\tbinom{n-i+1}{j}
\tbinom{m-j}{i-j}
X^\ast_{(1^j)}
\right)\\
&=&
\tfrac{(m-i+1) n(n+1)(n-1) (n-i+2)(n-m-i+1)}{i m (n-2i+2)(n-2i+3) (n-m)}
Z_{(1^{i-1})}.
\end{eqnarray*}
Therefore the desired result follows.
\end{proof}

\section{Proofs of main results}
\label{sec:proof}

In this section, 
we prove the next two propositions:

\begin{prop}\label{prop:tightE}
For any $\Ades$-design $X$ on $\G{m}{n}$, 
the inequality $|X| \geq \binom{n}{m}$ holds.
Furthermore,
let $S$ be a finite subset of $\G{m}{n}$ with $|S| = \binom{n}{m}$.
Then 
the following conditions on $S$ are equivalent:
\begin{enumerate}
\item $S$ is an $\Ades$-design on $\G{m}{n}$.
\item $y_m(a,b) = 0$ for any distinct $a,b \in S$.
\end{enumerate}
\end{prop}

\begin{prop}\label{prop:tightEF}
Let $S$ be a finite subset of $\G{m}{n}$ with $|S| = \binom{n}{m}$.
Then the following conditions on $S$ are equivalent:
\begin{enumerate}
\item $S$ is a $\Ades \cup \Adesh$-design on $\G{m}{n}$.
\item For any $a,b \in S$ and any $i=1,\dots,m$,
\[
y_i(a,b) = 0 \text{ or } 1.
\]
\end{enumerate}
\end{prop}

Theorem \ref{Mainthm:great_anti_to_E_design} and 
Theorem \ref{Mainthm:great_anti_to_EF_design}
are reduced to 
Proposition \ref{prop:tightE} and Proposition \ref{prop:tightEF}
as follows. 
By Fact \ref{fact:2-number_of_Gmn} and Proposition \ref{prop:equiv_antipodal_set}, 
a finite subset $S$ of $\G{m}{n}$ 
is great antipodal set 
if and only if 
$|S| = \binom{n}{m}$ and 
$y_i(a,b) = 0 \text{ or } 1$ 
for any $a,b \in S$ and any $i=1,\dots,m$.
In this case, 
$y_m(a,b) = 0$ for distinct $a,b \in S$ 
since $y_i(a,b) \geq y_{i+1}(a,b)$ and 
$y(a,b) \neq (1,\dots,1)$ by Observation \ref{obs:designs}.
Hence, Theorem \ref{Mainthm:great_anti_to_E_design} follows from Proposition \ref{prop:tightE}. 
Note that any $\Ades \cup \Adesh$-design $X$ is also a $\Ades$-design,
and hence $|X| \geq \binom{n}{m}$ by Theorem \ref{Mainthm:great_anti_to_E_design}.
Thus, 
Theorem~\ref{Mainthm:great_anti_to_EF_design} follows from Proposition~\ref{prop:tightEF}.

\subsection{Linear programming bounds}

In order to prove Proposition \ref{prop:tightE} and Proposition \ref{prop:tightEF},
we show the following proposition,
which is a kind of generalization of \cite[Theorem 9]{Roy2009bca}:

\begin{prop}\label{prop:LP-bounds}
Let \[
c : \Partition{m} \rightarrow \R,\quad \mu \mapsto c_{\mu}
\] be a real function on $\Partition{m}$ with the following properties:
\begin{enumerate}
\item $c_{(0)} > 0$.
\item $c$ has finite support, that is, 
\[
|\{\, \mu \in \Partition{m} \mid c_\mu \neq 0 \,\}| < \infty.
\]
\item The function
\[
F := \sum_{\mu \in \Partition{m}} c_\mu Z_\mu
\]
on $\Range{m}{n}$ is non-negative.
\end{enumerate}
We put 
\begin{align*}
\mathcal{T}_{c}^{+} := \{ \mu \in \Partition{m} \mid c_\mu > 0 \}, \\
\mathcal{T}_{c}^{-} := \{ \mu \in \Partition{m} \mid c_\mu < 0 \}.
\end{align*}
Then the following holds:
\begin{enumerate}
\item \label{item:LP:bounds} Let $X$ be a $\mathcal{T}_c^+$-design on $\G{m}{n}$.
Then $|X| \geq F(1,\dots,1)/c_{(0)}$.
\item \label{item:LP:eq_condition}
Let $X$ be a non-empty finite subset of $\G{m}{n}$.
If $X$ satisfies any two conditions in the following three conditions, then $X$ also satisfies the rest one:
\begin{description}
\item[Condition A:] $X$ is a $\mathcal{T}_c^+$-design on $\G{m}{n}$
\item[Condition B:] $X$ is a $\mathcal{T}_{c}^-$-design on $\G{m}{n}$ and \[
F(y(a,b)) = 0 \quad \text{for any distinct } a,b \in X.
\]
\item[Condition C:] $|X| = F(1,\dots,1)/c_{(0)}$.
\end{description}
\end{enumerate}
\end{prop}

\begin{proof}
Let $X$ be a non-empty finite subset of $\G{m}{n}$.
Since $F$ is non-negative, we have 
\[
\sum_{a,b \in X} F(y(a,b)) \geq \sum_{a \in X} F(y(a,a)) = |X| \cdot F(1,\dots,1).
\]
On the other hand, by the definition of $F$, we have
\begin{align*}
\sum_{a,b \in X} F(y(a,b)) = 
c_{(0)}|X|^2 +\sum_{a,b \in X} \sum_{\mu \neq (0)} c_{\mu} Z_\mu(y(a,b)).
\end{align*}
Therefore, we obtain 
\begin{align}
\sum_{\mu \in \mathcal{T}_c^+ \cup \mathcal{T}_c^- \setminus \{(0)\}} c_{\mu} \sum_{a,b \in X}Z_\mu(y(a,b)) \geq |X| (F(1,\dots,1) - c_{(0)}|X|) \label{eq:LP_proof}
\end{align}
and the equality holds if and only if 
\[
F(y(a,b)) = 0 \quad \text{for any distinct } a,b \in X.
\]

To prove Proposition \ref{prop:LP-bounds} \eqref{item:LP:bounds}, 
let us suppose that $X$ is a $\mathcal{T}_c^+$-design on $\G{m}{n}$.
Then by Proposition \ref{prop:eq_design_Z}, we have
\[
\sum_{a,b \in X} Z_\mu(y(a,b)) = 0 \quad \text{for any } \mu \in \mathcal{T}_c^+ \setminus \{(0)\}.
\]
Therefore, by combining this with \eqref{eq:LP_proof}, we have 
\begin{align}
\sum_{\mu \in \mathcal{T}_c^- \setminus \{0\}} c_{\mu} 
\sum_{a,b \in X}Z_\mu(y(a,b)) \geq |X| (F(1,\dots,1) - c_{(0)}|X|). \label{eq:ineqLP}
\end{align}
Here, by Fact \ref{fact:positive_Z}, the left hand side of \eqref{eq:ineqLP} is smaller than or equals to $0$,
and hence, we have 
\[
|X| \geq \frac{F(1,\dots,1)}{c_{(0)}}.
\]

We prove Proposition \ref{prop:LP-bounds} \eqref{item:LP:eq_condition} as follows:
\begin{description}
\item[If $A$ and $B$, then $C$:]
Suppose $X$ is a $\mathcal{T}_{c}^+ \cup \mathcal{T}_c^-$-design on $\G{m}{n}$ with 
\[
F(y(a,b)) = 0 \quad \text{for any distinct } a,b \in X.
\]
Then by \eqref{eq:LP_proof} and Proposition \ref{prop:eq_design_Z}, we have
\[
0 = |X| (F(1,\dots,1) - c_{(0)}|X|). 
\]
Since $X$ is non-empty, $|X| = F(1,\dots,1)/c_{(0)}$.
\item[If $B$ and $C$ then $A$:]
Suppose that $X$ is a $\mathcal{T}_c^-$-design on $\G{m}{n}$ with $|X| = F(1,\dots,1)/c_{(0)}$ and 
\[
F(y(a,b)) = 0 \quad \text{for any distinct } a,b \in X.
\]
Then by \eqref{eq:LP_proof} and Proposition \ref{prop:eq_design_Z}, we have
\[
\sum_{\mu \in \mathcal{T}_c^+ \setminus \{(0)\}} c_{\mu} 
\sum_{a,b \in X}Z_\mu(y(a,b)) = 0.
\]
By Fact \ref{fact:positive_Z}
and the positivity of $c_{\mu}$ for $\mu\in \mathcal{T}_c^+ \setminus \{(0)\}$,
we obtain
\[
\sum_{a,b \in X}Z_\mu(y(a,b)) = 0 \quad \text{for any } \mu \in \mathcal{T}_c^+ \setminus \{(0)\}.
\]
Thus, by Proposition \ref{prop:eq_design_Z}, 
$X$ is a $\mathcal{T}_c^+$-design on $\G{m}{n}$. 
\item[If $C$ and $A$ then $B$:] 
Suppose that $X$ is a $\mathcal{T}_{c}^+$-design on $\G{m}{n}$ with $|X| = \frac{F(1,\dots,1)}{c_{(0)}}$.
Then by \eqref{eq:LP_proof} and Proposition \ref{prop:eq_design_Z}, 
we have 
\[
\sum_{\mu \in \mathcal{T}_c^- \setminus \{(0)\}} c_{\mu} 
\sum_{a,b \in X}Z_\mu(y(a,b)) \ge 0.
\]
By Fact \ref{fact:positive_Z} and the negativity of $c_{\mu}$ for $\mu\in \mathcal{T}_c^- \setminus \{(0)\}$, 
we obtain
\[
\sum_{a,b \in X}Z_\mu(y(a,b)) = 0 \quad \text{for any } \mu \in \mathcal{T}_c^- \setminus \{(0)\}.
\]
Thus, by Proposition \ref{prop:eq_design_Z},
$X$ is a $\mathcal{T}_c^-$-design on $\G{m}{n}$. 
\end{description}
\end{proof}

\begin{rem}
Proposition \ref{prop:LP-bounds} is generalized to 
a theory for designs on general finite measure spaces in Okuda--Sawa \cite{OkudaSawa2013}.
\end{rem}

\subsection{Proof of Proposition \ref{prop:tightE}}

To apply Proposition \ref{prop:LP-bounds} to $\Ades$-designs, 
let us define an non-negative function $F^{\Ades}$ on $\Range{m}{n}$
by  
\begin{align*}
F^{\Ades} : \Range{m}{n} \rightarrow \mathbb{R}_{\geq 0},\quad (y_1,\dots,y_m) \mapsto 
\prod_{i=1}^{m} y_i.
\end{align*}
Then the following holds:

\begin{lem}\label{lem:F_E_can_be_written}
The function $F^{\Ades}$ can be written by 
\begin{align*}
F^{\Ades} = c^{\Ades}_{(0)} + \sum_{j=1}^m c^{\Ades}_{(1^j)} Z_{(1^j)} 
\end{align*}
with real-valued coefficients $c^\Ades_{\mu} \in \mathbb{R}$, $\mu \in \Ades$ such that  
\begin{enumerate}
\item $c^\Ades_{(0)} = 1/\binom{n}{m}$.
In particular, $F^{\Ades}(1,\dots,1)/c^{\Ades}_{(0)} = \binom{n}{m}$.
\item $c^{\Ades}_{(1^j)} > 0$ for any $j = 1,\dots,m$.
\end{enumerate}
\end{lem}

\begin{proof}
Note that $F^{\Ades} = X^\ast_{(1^i)}$.
Thus, our claim is in Proposition \ref{prop:X=sumZ}.
\end{proof}

We are ready to prove Proposition \ref{prop:tightE}:

\begin{proof}[Proof of Proposition \ref{prop:tightE}]
We extend $c^\Ades$, which is defined in Lemma \ref{lem:F_E_can_be_written}, to a function on $\Partition{m}$ 
by putting $c^\Ades_{\mu} = 0$ if $\mu \not \in \Ades$. 
Then, our first claim is proved 
by applying Proposition~\ref{prop:LP-bounds} \eqref{item:LP:bounds}
for $c = c^{\Ades}$.
Note that $\mathcal{T}_{c^\Ades}^+ = \Ades$ and $\mathcal{T}^{-}_{c^\Ades} = \emptyset$.
In order to prove the rest claims, 
let us fix a finite subset $S$ with $|S| = \binom{n}{m}$.
If $S$ is an $\Ades$-design,
then $S$ satisfies Condition A and Condition C in
Proposition~\ref{prop:LP-bounds} \eqref{item:LP:eq_condition} for $c = c^{\Ades}$.
This implies that 
\[
F^\Ades(y(a,b)) = 0 \quad \text{for any distinct } a,b \in S,
\]
and hence 
\begin{align}
y_m(a,b) = 0 \quad \text{for any distinct } a,b \in S, \label{eq:ymzero}
\end{align}
in this case.
Conversely, suppose that $S$ satisfies \eqref{eq:ymzero}.
Then $S$ satisfies Condition B and Condition C in
Proposition~\ref{prop:LP-bounds} for $c = c^\Ades$.
This implies that $S$ is an $\Ades$-design.
\end{proof}

\subsection{Proof of Proposition \ref{prop:tightEF}}

To apply Proposition \ref{prop:LP-bounds} to $\Ades \cup \Adesh$-designs, 
let us define an non-negative function $F^{\Adesh}$ on $\Range{m}{n}$ by 
\begin{align*}
F^{\Adesh} : \Range{m}{n} &\rightarrow \mathbb{R}_{\geq 0},\quad \\
(y_1,\dots,y_m) &\mapsto 
\binom{n-2}{m-1} (\prod_{i=1}^m y_i) (\sum_{i=1}^m y_i) + \sum_{i=1}^m y_i(1-y_i).
\end{align*}

\begin{lem}\label{lem:F_can_be_written}
The function $F^{\Adesh}$ can be written by 
\begin{align*}
F^{\Adesh} = c^{\Adesh}_{(0)}
+ \sum_{j=1}^m c^{\Adesh}_{(1^j)} Z_{(1^j)} + \sum_{j=2}^{m} c^{\Adesh}_{(2,1^{j-1})} Z_{(2,1^{j-1})} 
\end{align*}
with real-valued coefficients $c^\Adesh_{\mu} \in \mathbb{R}$, $\mu \in \Adesh$ such that  
\begin{enumerate}
\item $c^{\Adesh}_{(0)} = m \binom{n-2}{m-1}/\binom{n}{m}$. 
In particular, $F^{\Adesh}(1,\dots,1)/c^{\Adesh}_{(0)} = \binom{n}{m}$.
\item $c^{\Adesh}_{(2,1^{j-1})} > 0$ for any $j = 2,\dots,m$. 
\end{enumerate}
\end{lem}

\begin{proof}
Let $d^{(i)}_j$ be the coefficient of $Z_{(1^j)}$ in \eqref{eq:X=sumZ}.
By the definition of the Schur polynomials,
we obtain $\prod^m_{i=1} y_i=X^\ast_{(1^m)}$, $\sum^m_{i=1} y_i=m X^\ast_{(1)}$ and $\sum^m_{i=1} y_i^2=\binom{m+1}{2}X^\ast_{(2)}-\binom{m}{2}X^\ast_{(1,1)}$.
Using Proposition~\ref{prop:X=sumZ}, Proposition~\ref{prop:zonal_poly_(2111)}
for $i=1$ and Lemma~\ref{lem:Z(1)Z(1^i)},
we have that $F^{\Adesh}$ is written by a linear combination of
$\{Z_{(1^j)}\}^m_{j=0}$ and $\{Z_{(2,1^{j-1})}\}^m_{j=1}$
with real-valued coefficients.
In particular 
we can check that the coefficients of $Z_{(0)}$, $Z_{(2)}$
and $Z_{(2,1^{j-1})}$
are $c^{\Adesh}_{(0)} = m \binom{n-2}{m-1}/\binom{n}{m}=\frac{m^2 (n-m)}{n (n-1)}$,
$c^{\Adesh}_{(2)}=0$ and $c^{\Adesh}_{(2,1^{j-1})}=d^{(m)}_j d^{(1)}_1 a_j m \binom{n-2}{m-1}>0$, where $a_j$ is the positive number which appears in Lemma~\ref{lem:Z(1)Z(1^i)},
respectively.
\end{proof}

\begin{proof}[Proof of Proposition \ref{prop:tightEF}]
We extend $c^\Adesh$, which is defined in Lemma \ref{lem:F_can_be_written},
to a function on $\Partition{m}$ 
by putting $c^{\Adesh}_{\mu} = 0$ if $\mu \not \in \Ades \cup \Adesh$. 
By Proposition \ref{prop:LP-bounds} \eqref{item:LP:eq_condition}, 
for a finite subset $S$ of $\G{m}{n}$ with $|S| = \binom{n}{m}$ the following conditions on $S$ are equivalent:
\begin{enumerate}
\item $S$ is a $\mathcal{T}_{c^\Adesh}^+$-design on $\G{m}{n}$.
\item $S$ is a $\mathcal{T}_{c^\Adesh}^-$-design with 
\begin{align}
F^\Adesh(y(a,b)) = 0 \quad \text{for any distinct } a,b \in S \label{eq:F^F=0}
\end{align}
\end{enumerate}
Here, by Lemma \ref{lem:F_can_be_written}, 
we have 
\begin{align*}
\Adesh \subset \mathcal{T}_{c^\Adesh}^+ \subset \Ades \cup \Adesh, \quad \mathcal{T}_{c^\Adesh}^- \subset \Ades.
\end{align*}
Let us assume that $S$ is a $\Ades \cup \Adesh$-design.
Then $S$ is also a $\mathcal{T}_{c^\Adesh}^+$-design.
By Observation~\ref{obs:y=111or000} \eqref{obs:y=111},
$y(a,a) = (1,\dots,1)$ for any $a \in \G{m}{n}$.
Therefore, $S$ satisfies \eqref{eq:F^F=0} and hence,
\begin{align}
y_i(a,b) = 0 \text{ or } 1 \quad \text{ for any } a,b \in S, \label{eq:y=01}
\end{align}
Conversely,
suppose that $S$ satisfies \eqref{eq:y=01}.
Then by Proposition \ref{prop:tightE}, 
the set $S$ is an $\Ades$-design on $\G{m}{n}$.
Hence, $S$ is a $\mathcal{T}_{c^\Adesh}^-$-design with \eqref{eq:F^F=0}.
Therefore, $S$ is also a $\mathcal{T}_{c^\Adesh}^+$-design.
This implies that $S$ is a $\Ades \cup \Adesh$-design on $\G{m}{n}$.
\end{proof}

\appendix
\section{Lower bounds for $1$-designs on $\G{m}{n}$}
\label{sec:appendix}

Let $t$ be a non-negative integer.
As an analogy of
$t$-designs on rank one symmetric spaces,
the concept of $t$-designs on $\G{m}{n}$ was introduced by
Roy~\cite{Roy2009bca}.
Remark that a $t$-design on $\G{m}{n}$ in terms of Roy's definition
is translated as a $\mathcal{T}_t$-design on $\G{m}{n}$,
where $\mathcal{T}_t:=\set{\mu\in \Partition{m}}
{\sum^m_{i=1}\mu_i\le t}$.
\begin{rem}
A great antipodal set $S$ is a $1$-design on $\G{m}{n}$
by Theorem~\ref{Mainthm:great_anti_to_E_design}.
However $S$ can not be a $2$-design in general.
In fact, we can check that $\sum_{a,b\in S} Z _{(2)}(y(a,b))\neq 0$
by Remark~\ref{rem:Z_i},
that is, $S$ is not a $\{(2)\}$-design.
\end{rem}

Roy~\cite[Lemma 9]{Roy2009bca} gave
the lower bound for the cardinalities of 
$t$-designs on $\G{m}{n}$ as follows:
any $t$-design $X$ satisfies
\[
|X| \ge \sum_{\mu\in \mathcal{T}_{\lfloor t/2\rfloor}}\dim H_\mu .
\]
Note that
since
$\sum_{\mu\in \mathcal{T}_{\lfloor 1/2\rfloor}}\dim H_\mu=
\dim H_{(0)}=1$ holds,
the above bound for 1-designs becomes trivial.
In this section, we give a sharper lower bound for 1-designs on $\G{m}{n}$ than the above bound.
In particular, we determine $1$-designs on $\G{m}{n}$
with the smallest cardinalities when $n$ is divided by
$m$.
Essentially,
the proof of the following theorem due to Theorem 9 in
Roy~\cite{Roy2009bca}.

\begin{thm}
\label{thm:Roy_1_design}
Let $X$ be a $1$-design on $\G{m}{n}$.
Then the inequality
\[
|X| \ge \frac{n}{m}.
\]
holds.
Furthermore equality attained if and only if
$m | n$ and
$X$ is consist of $m$-dimensional subspaces
$\{x_1,x_2,\ldots ,x_{n/m}\}$ such that
\[
\C^n=x_1 + x_2 +\cdots +x_{n/m}
\ 
\text{and $x_i \perp x_j$ if $i\neq j$.}
\]
\end{thm}
\begin{rem}
\label{rem:partition_Cn_AP}
By Observation~\ref{obs:y=111or000} \eqref{obs:y=000},
the above condition of $\{x_i\}^{n/m}_{i=1}$
is equivalent to 
$y(x_i,x_j)=(0,0,\ldots ,0)$ if $i\neq j$.
By Proposition~\ref{prop:equiv_antipodal_set},
$\{x_i\}^{n/m}_{i=1}$ is an antipodal set.
\end{rem}
\begin{proof}[Proof of Theorem~\ref{thm:Roy_1_design}]
Consider $X^\ast_{(1)}(y)=(\sum^m_{i=1}y_i)/m$.
From Proposition~\ref{prop:X=sumZ},
we obtain $X^\ast_{(1)}=\frac{m}{n}Z_{(0)}+
\frac{n-m}{n (n-1) (n+1)} Z_{(1)}$.
Let the $\R$-valued function $c$ on $\Partition{m}$
be $c_{(0)}=\frac{m}{n}$,
$c_{(1)}=\frac{n-m}{n (n-1) (n+1)}$
and $c_{\mu}=0$ otherwise.
Then $\mathcal{T}_1=\mathcal{T}^+_{c}$.
By Proposition~\ref{prop:LP-bounds},
for a 1-design $X$, we have $|X|\ge X^\ast_{(1)}(1,1,\ldots ,1)/c_{(0)}=n/m$
and that equality attained if and only if
$m | n$ and
$X^\ast_{(1)}(y(a,b))=0$ for $a,b\in X$ with $a\neq b$.
Since for each $a,b\in X$ and $i=1,2,\ldots ,m$,
$y_i(a,b)$ is non-negative,
$X^\ast_{(1)}(y(a,b))=0$ implies $y_i(a,b)=0$
for $i=1,2,\ldots ,m$.
By Observation~\ref{obs:y=111or000} \eqref{obs:y=000},
the desired result follows.
\end{proof}

\section{An $\Ades$-design with the smallest cardinality}
\label{sec:appendix2}

In this section, we give an example of an $\mathcal{E}$-design 
with the smallest cardinality which is not a great antipodal set.

Let $\{e_1,e_2,e_3,e_4\}$
denotes the standard orthonormal basis of $\C^4$.
We consider the following six spaces:
\[
x_1:=\C \vspan \{ e_1,e_2 \},\ 
x_2:=\C \vspan \{ e_3,e_4 \},
\]
\[
x_3:=\C \vspan \{ e_1,e_4 \},\ 
x_4:=\C \vspan \{ e_2,e_4 \},
\]
\[
x_5:=\C \vspan \{\, e_1+\sqrt{-1} e_2, e_3 \,\},\ 
x_6:=\C \vspan \{\, e_1-\sqrt{-1} e_2, e_3 \,\}
\]
and $X:=\{x_1,x_2,\ldots ,x_6\}$.
Then $X$ is a subset of $\G{2}{4}$ which is consist of six points.
On the other hand, we can check that 
the matrix whose the $(i,j)$-entry is $y(x_i,x_j)$
for $1\le i,j\le 6$
is 
\[
\left(
\begin{array}{cccccc}
  (1,1) & (0,0) & (1,0) & (1,0) & (1,0) & (1,0) \\ 
  (0,0) & (1,1) & (1,0) & (1,0) & (1,0) & (1,0) \\ 
  (1,0)  & (1,0) & (1,1)  & (1,0)  & (1/2,0) & (1/2,0) \\ 
  (1,0)  & (1,0) & (1,0)  & (1,1)  & (1/2,0) & (1/2,0) \\ 
  (1,0) & (1,0)  & (1/2,0) & (1/2,0)  & (1,1)  & (1,0)  \\ 
  (1,0) & (1,0)  & (1/2,0) & (1/2,0)  & (1,0)  & (1,1)  \\ 
\end{array}
\right).
\]
From the above matrix, we can see the following two facts.
The principal angles $y(x_3,x_5)$ 
between $x_3$ and $x_5$ coincides with
$(1/2,0)$, that is,
$X$ is not an antipodal set by Proposition~\ref{prop:equiv_antipodal_set}.
On the other hand,
any last principal angle $y_2(x_i,x_j)$ for $i\neq j$ 
is zero and $|X|=\binom{4}{2}=6$, that is,
$X$ is an $\mathcal{E}$-design 
by Proposition~\ref{prop:tightE}.
Thererfore $X$ is an example of a `tight' $\mathcal{E}$-design which is not a great antipodal set.

Remark that 
$\{x_1,x_2,x_3,x_4\}$ and $\{x_5,x_6\}$ are antipodal sets.
Hence $X$ is a disjoint union of two antipodal sets.

\noindent
\textbf{Acknowledgements.}
The authors would like to thank Makiko Sumi Tanaka and Hiroyuki Tasaki whose comments were of inestimable value for our study. 
We are also indebted to Kohei Suzuki for his careful reading of this paper.
The first author is supported by 
GCOE `Fostering top leaders in mathematics',
Kyoto University, and
the second author is supported by the fellowship of the Japan Society for
the Promotion of Science.

\quad\\
{\it Hirotake Kurihara}\\
	Research Institute for Mathematical Sciences, \\
	Kyoto University, \\
	Kyoto 606-8502,\\
	Japan\\
	\texttt{kurihara@kurims.kyoto-u.ac.jp}\\
\quad\\
{\it Takayuki Okuda}\\
	Graduate School of Mathematical Science,\\
	The University of Tokyo,\\
	3-8-1 Komaba,\\
	Meguro-ku,\\
	Tokyo 153-8914,\\
	Japan\\
	\texttt{okuda@ms.u-tokyo.ac.jp}

\end{document}